\documentclass{amsart}

\usepackage{amsthm}
\usepackage{amsfonts}
\usepackage{amsmath, amssymb}
\usepackage{enumerate}
\usepackage{graphicx}
\usepackage{caption} %%% optional but probably necessary
\usepackage{subcaption} %%% optional but probably necessary
\usepackage{tikz}
\usetikzlibrary{matrix,arrows,decorations.pathmorphing}
\usetikzlibrary{knots}
\theoremstyle{definition}
\newtheorem{define}{Definition}[section]
\newtheorem{remark}{Remark}[section]
\newtheorem{theorem}{Theorem}[section]
\newtheorem{lemma}{Lemma}[section]

\newtheorem{corollary}{Corollary}[section]

\makeatletter
\newtheorem*{rep@theorem}{\rep@title}
\newcommand{\newreptheorem}[2]{%
\newenvironment{rep#1}[1]{%
 \def\rep@title{#2 \ref{##1}}%
 \begin{rep@theorem}}%
 {\end{rep@theorem}}}
\makeatother

\newreptheorem{theorem}{Theorem}
\newreptheorem{corollary}{Corollary}

\begin{document}

%%%%%%%%%%%%%%
%% Abstract %%
%%%%%%%%%%%%%%
\begin{abstract}
For a knot $K\subset S^3$, its exterior $E(K) = S^3\backslash\eta(K)$ has a singular foliation by Seifert surfaces of $K$ derived from a circle-valued Morse function $f\colon E(K)\to S^1$. When $f$ is self-indexing and has no critical points of index 0 or 3, the regular levels that separate the index-1 and index-2 critical points decompose $E(K)$ into a pair of compression bodies. We call such a decomposition a \textit{circular Heegaard splitting} of $E(K)$. We define the notion of \textit{circular distance} (similar to Hempel distance) for this class of Heegaard splitting and show that it can be bounded under certain circumstances. Specifically, if the circular distance of a circular Heegaard splitting is too large: (1) $E(K)$ can't contain low-genus incompressible surfaces, and (2) a minimal-genus Seifert surface for $K$ is unique up to isotopy.
\end{abstract}

%%%%%%%%%%%%%%%%%%
%% Title Matter %%
%%%%%%%%%%%%%%%%%%
\title{A Distance for Circular Heegaard Splittings}

\author{Kevin Lamb}
\address{Kevin Lamb, Department of Mathematics, University of the Pacific}
\email{klamb@pacific.edu}

\author{Patrick Weed}
\address{Patrick Weed, Department of Mathematics, University of California, Davis}
\email{psweed@math.ucdavis.edu}
\maketitle

%\begin{center}
%Kevin Lamb\footnote{Department of Mathematics, \latex{University of the Pacific} - \texttt klamb@pacific.edu\normalfont}
%\and
%Patrick Weed\footnote{Department of Mathematics, University of California, Davis}
%\end{center}

%%%%%%%%%%%%%%%%%%%%%%%%%%%%%%%%%%%%%%%%%%%%%%%%%%%%%%%%%%%%%%%%%%%%%%%%%%%%%%%%%%%%%%%%%%%%%%%%%%%%%%%%%%%%%%%%%%%%%
%%%%%%%%%%%%%%%%%%%%%%%%%%%%%%%%%%%%%%%%% B E G I N   P A P E R %%%%%%%%%%%%%%%%%%%%%%%%%%%%%%%%%%%%%%%%%%%%%%%%%%%%%
%%%%%%%%%%%%%%%%%%%%%%%%%%%%%%%%%%%%%%%%%%%%%%%%%%%%%%%%%%%%%%%%%%%%%%%%%%%%%%%%%%%%%%%%%%%%%%%%%%%%%%%%%%%%%%%%%%%%%

%%%%%%%%%%%%%%%%%%
%% Introduction %%
%%%%%%%%%%%%%%%%%%
\section{Introduction}

In the study of classical knot theory, it is often the case that questions about a knot can be translated into questions about a 3-manifold (and vice-versa). Hence, we can study knots with tools from the theory of 3-manifolds. In particular, for a knot $K$ in a 3-manifold $M$, we study its \textit{complement} $C(K) = M\backslash K$ or its (compact) \textit{exterior} $E(K) = M\backslash\eta(K)$, where $\eta(K)$ is an open regular neighborhood of $K$ in $M$. The classification of 3-manifolds is a central problem in low-dimensional topology, and decomposition theorems lie at its heart.

Decompositions of 3-manifolds are typically created via surfaces. Hierarchies are a primary example of this methodology (cf. \cite{Jaco}), although the present work concerns itself with \textit{Heegaard splittings} of 3-manifolds. This is a decomposition of a 3-manifold $M$ into two \textit{handlebodies} $H_1, H_2$ via a closed separating surface $\Sigma$ (called a \textit{Heegaard surface}) such that $H_1\cap H_2 = \partial H_1 = \partial H_2 = \Sigma$ and $H_1\cup H_2 = M$. By a theorem of Moise \cite{Moise}, every closed, orientable 3-manifold admits a triangulation $\mathcal{T}$. We can then construct a Heegaard splitting of $M$ using the second barycentric subdivision of $\mathcal{T}$. Hence, every closed, orientable 3-manifold has a Heegaard splitting.

Although Heegaard surfaces separate a 3-manifolds into standard components, they aren't ``elementary'' surfaces of 3-manifolds. A Heegaard surface $\Sigma$ for a 3-manifold $M$ can be \textit{compressed} (via properly embedded \textit{compression disks}) into a less complicated surface through standard cut-and-paste techniques. In fact, $\Sigma$ can be compressed to a 2-sphere $S^2$ on both sides (though, not necessarily simultaneously) since $\Sigma$ bounds handlebodies to both sides. A more useful property for a surface to have is to not compress entirely to a 2-sphere. In other words, we'd like for there to be a surface $F\subset M$ that can't be compressed (to either side if $F$ is orientable). Such a surface is called \textit{incompressible} and can be thought of as an elementary part of $M$.

Although they aren't incompressible surfaces, we can study how compressible Heegaard surfaces are to both sides simultaneously. Harvey \cite{Harv} described a way to assign a simplicial complex to a compact, orientable surface $F$. The vertices of this complex are isotopy classes of simple closed curves on $F$ that do not bound disks in $F$ and are not parallel into $\partial F$ (we call these curves \textit{essential} in $F$). We draw an edge between two distinct vertices if they have representatives that are disjoint on $F$. The remainder of this simplicial complex $\mathcal{C}(F)$ is built as a flag complex (filling in simplices for vertices with representatives in $F$ that are pairwise disjoint from one another in $F$), though we will only concern ourselves with its 1-skeleton $C(F)$.

For a Heegaard splitting $H_1\cup_\Sigma H_2 = M$, Hempel \cite{Hemp} defined a measure of incompressibility of $\Sigma$ in $M$. A pair of compression disks $D_i\subset H_i$ ($i=1,2$) have boundary curves that lie on $\Sigma$. These curves are representatives of vertices in $C(\Sigma)$, and we can use the graph distance to measure how far separated they are in $C(\Sigma)$. The minimum distance in $C(\Sigma)$ between the boundaries of all pairs of compression disks $D_1\subset H_1$, $D_2\subset H_2$ is called the \textit{(Hempel) distance} of $\Sigma$, denoted $d(\Sigma)$. We can define standard classifications of Heegaard splittings using this notion. We say $M$ is: (1) \textit{reducible} if $d(\Sigma) = 0$, (2) \textit{weakly reducible} if $d(\Sigma) = 1$, and (3) \textit{strongly irreducible} if $d(\Sigma)\geq2$. Distance has been used in many settings (cf. \cite{BS}, \cite{Bir}, \cite{E}, \cite{Thom}) to prove various results both algebraic and geometric in nature.

In \cite{Hart}, Hartshorn showed that incompressible surfaces can be used to bound the distance of Heegaard splittings of closed, orientable 3-manifolds. In particular, he shows that for any 3-manifold $M$ that contains an incompressible surface of genus $g\geq1$, the distance $d(\Sigma)$ of any Heegaard splitting $\Sigma$ of $M$ is bounded by $2g$. Scharlemann and Tomova \cite{ScTom} elaborate on this idea and use Cerf theoretical arguments to show that there is a similar bound if we replace the incompressible surface with a strongly irreducible Heegaard splitting $P$; that is, $d(\Sigma)\leq2-\chi(P) = 2g(P)$. We recall the notion of distance of a Heegaard splitting and state these theorems precisely later.

Let $K\subset S^3$ be a knot. In \cite{Fab}, Manjarrez-Guti\'errez employed a circle-valued Morse function $f\colon E(K)\to S^3$ to induce a \textit{circular handle decomposition} of the exterior $E(K)$ of $K$. She then adapted the ideas of \cite{ST2} to create a notion of \textit{circular thin position} of such decompositions of $E(K)$. We elaborate on these ideas and develop notions of \textit{circular Heegaard splittings} and of \textit{circular distance} of circular Heegaard splittings.

Circular Heegaard splittings are different from the usual Heegaard splittings of 3-manifolds in two important ways: (i) the Heegaard surface is not closed and (ii) the Heegaard surface is comprised of two connected components. Both of these pose a challenge to the definition of a circular distance. We show that methods of Hartshorn carry over to this setting, but much care needs to be taken in their adaptation. We show how the shift from Heegaard splittings of closed manifolds to circular Heegaard splittings of knot exteriors restricts these methods. We also provide analogous technical lemmas for this new setting. Once these have laid a solid foundation for Hartshorn-like methods, we define the circular distance of a circular Heegaard splitting and show that it behaves as expected using these methods.

We then prove the main theorem:

\begin{reptheorem}{thm:Main}[Main Theorem]
Let $K\subset S^3$ be a knot whose exterior $E(K)$ has a circular Heegaard splitting $(F,S)$ such that $F$ is incompressible. If its exterior $E(K)$ contains a closed, orientable, essential surface $G$ of genus $g$, then $cd(F,S) \leq 2g$.
\end{reptheorem}

An immediate corollary follows and bears striking resemblance to Hartshorn's theorem for closed 3-manifolds; if we consider essential surfaces disjoint from $F$, we get a bound in the usual curve complex for $S$:

\begin{repcorollary}{cor:Main}
Let $K\subset S^3$ be a knot whose exterior $E(K)$ has a circular Heegaard splitting $(F,S)$ such that $F$ is incompressible. If its exterior $E(K)$ contains a closed, orientable, essential surface $G$ of genus $g$ that can be isotoped to be disjoint from $F$, then $d(S) \leq 2g$.
\end{repcorollary}

We can show a similar theorem for the case when the incompressible surface has that knot $K$ as its boundary. That is, we modify Theorem \ref{thm:Main} to account for the case when $G$ is an incompressible Seifert surface for $K$:

\begin{reptheorem}{thm:SeifertMain}
Let $K\subset S^3$ be a knot whose exterior $E(K)$ has a circular Heegaard splitting $(F,S)$ such that $F$ is incompressible. If $F'$ is an incompressible Seifert surface for $K$ with genus $g$ that is not isotopic to $F$, then $d(S)\leq2g+1$.
\end{reptheorem}

This result leads to a partial affirmation of a conjecture of Manjarrez-Guti\'errez \cite{Fab} pertaining to the thin levels of circular thin positions of $E(K)$ and their relation to minimal-genus Seifert surfaces of $K$:

\begin{repcorollary}{cor:UniqueSeifert}
Let $K\subset S^3$ be a knot whose exterior $E(K)$ has a circular Heegaard splitting $(F,S)$ such that $F$ is a Seifert surface of $K$ that realizes the Seifert genus $g(K)$. If $d(S)>2g(K)+1$, then $F$ is the unique Seifert surface of minimal genus for $K$ up to isotopy.
\end{repcorollary}

This presentation herein was submitted as the core work in the doctoral dissertation of the first author, but the results within are the core work of the dissertations of both authors. Both authors would like to thank their adviser, Abby Thompson, for her guidance in this work.

%%%%%%%%%%%%%%%%%%%%%%%%%%%%%%%%
%% Background and Definitions %%
%%%%%%%%%%%%%%%%%%%%%%%%%%%%%%%%
\section{Background and Definitions}

%%%%%%%%%%%%%%%%%%%%%
%% Knots and Links %%
%%%%%%%%%%%%%%%%%%%%%
\subsection{Knots and Links}

For a positive integer $n$, let $S_n = S^1\sqcup\cdots\sqcup S^1$ be a disjoint union of $n$ circles. Let $M^3$ be a compact 3-manifold. Two smooth embeddings $f,g\colon S_n\to M$ are called \textit{isotopic} if there exists a homotopy $h\colon S^1\times[0,1]\to M$ such that

\begin{enumerate}

	\item
	$f = h|_{S_n\times\{0\}}$
	
	\item
	$g = h|_{S_n\times\{1\}}$
	
	\item
	for all $0\leq t\leq1$, $h|_{S_n\times\{t\}}$ is an embedding.

\end{enumerate}
We define a \textit{link in $M$} to be the isotopy class of a smooth embedding $f\colon S_n\to M$, and we set $n$ to be the number of \textit{components} of the link. A \textit{knot in $M$} is defined to be a link with one component.

For a link $L\subset M$, the 3-manifold $C(L) = M\backslash L$ is called the \textit{complement of $L$ in $M$}. We distinguish this from the compact \textit{exterior of $L$ in $M$}, denoted by $E(L) = M\backslash\eta(L)$.

For a link $L\subset M$, a compact, orientable surface $F\subset M$ whose boundary $\partial F$ is the link will be called a \textit{Seifert surface} of $L$. It should be noted that such a surface always exists for any link in $S^3$.

A \textit{handlebody} is closed regular neighborhood of a graph $\Gamma$. A graph that is a deformation retract of a handlebody will be called a \textit{spine} of the handlebody. See Figure \ref{fig:HandlebodyCompressionBody}.

A link $L\subset M$ is said to be \textit{fibered} if there exists a fibration $f\colon S^3\backslash L\to S^1$ such that each component $L_i$ has a neighborhood framed as $S^1\times D^2$ (with $L_i\cong S^1\times\{0\}$) where $f_{S^1\times(D^2\backslash\{0\})}$ is the map $(x,y)\mapsto\frac{y}{|y|}$. This framing condition is in place to specify behavior near $L$ since $M\backslash L$ is not compact.

%%%%%%%%%%%%%%%%%%%%%%%%%%%%%%%%%%%%%%%%
%% 3-Manifolds and Heegaard Splitings %%
%%%%%%%%%%%%%%%%%%%%%%%%%%%%%%%%%%%%%%%%
\subsection{3-Manifolds and Heegaard Splittings}

\subsubsection{Compressibility and Irreducibility}

Let $M$ be a 3-manifold and $F\subset M$ be a compact, orientable surface properly embedded in $M$. We say that \textit{$F$ is compressible in $M$} if there exists a disk $D\subset M$ to either side of $F$ such that

\begin{enumerate}

	\item
	$F\cap D = \partial D\subset F$,
	
	\item
	$\partial D$ is not $\partial$-parallel into $\partial F$, and
	
	\item
	$\partial D$ does not bound a disk in $F$; that is, $\partial D$ is a curve \textit{essential in $F$}.
	
\end{enumerate}
Such a disk $D$ is called a \textit{compressing disk for $F$ in $M$}. See Figure \ref{fig:Compression}. If there are no compressing disks for $F$ in $M$, we say that $F$ is \textit{incompressible in $M$}. An incompressible surface $F$ is called \textit{essential in $M$} if it also not $\partial$-parallel into $\partial M$.

%%%%%%%%%%%%%%%%%%%%%%%
%% Compression Image %%
%%%%%%%%%%%%%%%%%%%%%%%
\begin{figure}[h]
\centering
\includegraphics[width=5in]{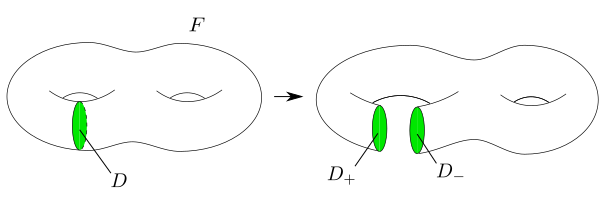}
\caption{A compression of the surface $F$ via the compressing disk $D$.}
\label{fig:Compression}
\end{figure}

Let $\eta(D)\subset M$ be an open regular neighborhood of a compressing disk $D$ for a surface $F\subset M$. Denote by $D_+$ and $D_-$ the two disk components of $\partial\eta(D)$ so that $A = \partial\eta(D)\backslash(D_+\cup D_-)\subset F$ is an annulus. Then the surface $(F\backslash A)\cup(D_+\cup D_-)$ will be called a \textit{compression of $F$ in $M$}. The process of replacing $F$ with a compression of $F$ will be called \textit{compressing $F$ in $M$}.

If $M$ has non-empty boundary and for a compact surface $F$ we have $(F,\partial F)\subset(M,\partial M)$, we say $F$ is \textit{properly embedded in $M$} if both $F$ is embedded in $M$ and $\partial F$ is embedded in $\partial M$. A properly embedded surface $F\subset M$ is said to be \textit{$\partial$-compressible} if there exists a disk $D\subset M$ such that

\begin{enumerate}

	\item
	$D\cap F = \alpha\subset\partial D$ and $D\cap\partial M = \beta\subset\partial D$,
	
	\item
	$\alpha\cup\beta = \partial D$ with $\alpha\cap\beta = \partial\alpha = \partial\beta$, and
	
	\item
	$\alpha$ does not cobound a disk with another arc in $\partial F$; that is, $\alpha$ is an arc \textit{essential in $F$}.
	
\end{enumerate}
Such a disk $D$ above is called a \textit{$\partial$-compressing disk} for $F$ in $M$. See Figure \ref{fig:BoundaryCompression} If there are no $\partial$-compressing disks for $F$ in $M$, we say that $F$ is \textit{$\partial$-incompressible in $M$}.

%%%%%%%%%%%%%%%%%%%%%%%%%%%%%%%%
%% Boundary Compression Image %%
%%%%%%%%%%%%%%%%%%%%%%%%%%%%%%%%
\begin{figure}[h]
\centering
\includegraphics[width=5in]{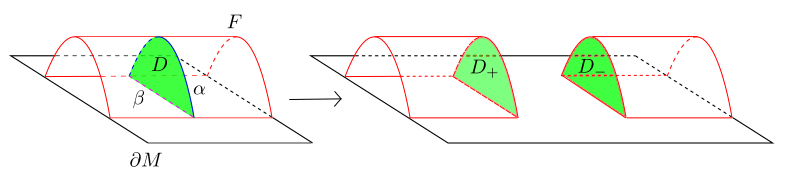}
\caption{A $\partial$-compression of the surface $F$ via the compressing disk $D$.}
\label{fig:BoundaryCompression}
\end{figure}

Let $\eta(D)\subset M$ be an open regular neighborhood of a $\partial$-compressing disk $D$ for a surface $F\subset M$. Denote by $D_+$ and $D_-$ the two disk components of $\partial\eta(D)$ parallel to $D$ so that $\partial\eta(D)\backslash(D_+\cup D_-) = \eta(\alpha)\cup\eta(\beta)$ is an annulus. Then the surface $(F\backslash\eta(\alpha)\cup(D_+\cup D_-)$ will be called a \textit{$\partial$-compression of $F$ in $M$}. The process of replacing $F$ with a $\partial$-compression of $F$ will be called \textit{$\partial$-compressing $F$ in $M$}.

A 3-manifold $M$ is called \textit{irreducible} if every $S^2\subset M$ bounds a 3-ball to at least one side. We say an irreducible 3-manifold $M$ is \textit{Haken} if it contains an incompressible, orientable surface $F$ with positive genus.

\begin{theorem}[Asphericity]
\label{thm:Asphericity}
For any knot $K\subset S^3$, $\pi_n(E(K))$ is trivial for $n\geq2$.
\end{theorem}

This follows from the sphere theorem and a theorem of Whitehead. A full proof can be found in \cite{Rolf}. In particular, the Asphericity Theorem tells us that knot complements (and exteriors) in $S^3$ contain no essential 2-spheres; that is, for a knot $K\subset S^3$, both $C(K)$ and $E(K)$ are irreducible 3-manifolds.

\subsubsection{Morse Theory and Heegaard Splittings}

Let $0\leq k\leq n$ be integers. An \textit{$n$-dimensional $k$-handle} is an $n$-ball $B^n = B^k\times B^{n-k}$. When the dimension is immaterial or clear from context, we abbreviate this to a \textit{$k$-handle}. The term $B^{k}\times\{0\}$ will be called the \textit{core} of a $k$-handle, and the term $\{0\}\times B^{n-k}$ will be called its \textit{co-core}. We see in the boundary
\[
S^{n-1} = \partial B^n = [(\partial B^k)\times B^{n-k}]\cup[B^k\times(\partial B^{n-k})] = [S^{k-1}\times B^{n-k}]\cup[B^k\times S^{n-k-1}].
\]
The term $S^{k-1}\times B^{n-k}$ is referred to as the \textit{attaching region} of the $k$-handle. In the case $k=0$, we take $S^{-1} = \emptyset$ and ``attaching a 0-handle'' is the same as adding a disjoint $B^n$. In our context, a handlebody as described in the section above will be regarded as a single 3-dimensional 0-handle with 1-handles glued to its boundary along their attaching regions. If we connect via radial line segements the cores of the 1-handles to the core of the 0-handle, the resulting graph may be regarded as a spine of the handlebody. See Figure \ref{fig:CoresAsSpine}. The four types of 3-dimensional handles are shown in Figure \ref{fig:3DHandles}.

%%%%%%%%%%%%%%%%%%%%%%%%%%
%% Cores As Spine Image %%
%%%%%%%%%%%%%%%%%%%%%%%%%%
\begin{figure}[h]
\centering
\includegraphics{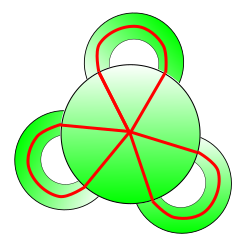}
\caption{Using the cores of 1-handles to define a spine of a 3-dimensional handlebody.}
\label{fig:CoresAsSpine}
\end{figure}

%%%%%%%%%%%%%%%%%%%%%%
%% 3D Handles Image %%
%%%%%%%%%%%%%%%%%%%%%%
\begin{figure}[h]
\centering
\includegraphics[width=5in]{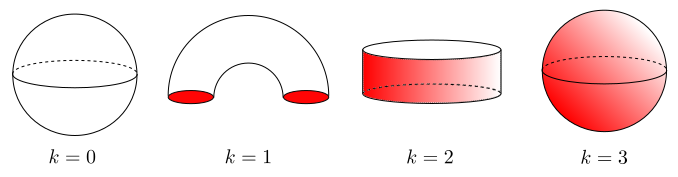}
\caption{The four types of 3-dimensional handles and their attaching regions (in red).}
\label{fig:3DHandles}
\end{figure}

Let $F$ be a compact, connected, orientable surface. Define $W_0 = (F\times I)\cup T$, where $T$ is a collection of 2-handles attached to $F\times\{1\}$. If there are any $S^2$ components in $\partial W_0$, we fill them with 3-handles. The resulting 3-manifold $W$ is called a \textit{compression body}. We denote by $\partial_-W$ the surface $F\times \{0\}$ and $\partial_+W = \partial W\backslash \partial_-W$. The union of a graph and the inner boundary $\partial W\backslash(F\times\{0\})$ that is a deformation retract of $W$ is called a \textit{spine} of $W$. A handlebody is considered a trivial compression body in the sense that $\partial_+W = \emptyset$. See Figure \ref{fig:HandlebodyCompressionBody}.

%%%%%%%%%%%%%%%%%%%%%%%%%%%%%%%%%%%%%%%%%%%
%% Handlebody and Compression Body Image %%
%%%%%%%%%%%%%%%%%%%%%%%%%%%%%%%%%%%%%%%%%%%
\begin{figure}[h]
\centering
\includegraphics[width=5in]{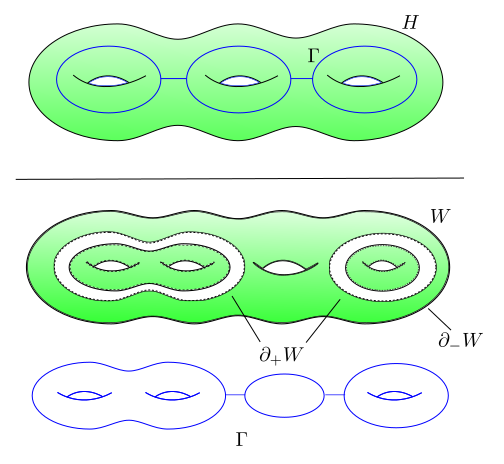}
\caption{A handlebody $H$ and a compression body $W$ with representative spines for each.}
\label{fig:HandlebodyCompressionBody}
\end{figure}

Let $M$ be a compact, orientable $n$-manifold. A smooth function $f\colon M\to[0,1]$ is called \textit{Morse} if all of its critical points are non-degenerate and occur at different critical values. The non-degeneracy condition implies that at each critical point $p\in M$ of $f$ there is a coordinate patch $x = (x_1,\ldots,x_n)$ with $p$ identified with the origin such that
\[
f(x) = f(p)-x_1^2-\cdots-x_k^2+x_{k+1}^2+\cdots+x_n^2.
\]
We call $k\in\{0,1,\ldots,n\}$ the \textit{index} of the critical point $p$. Given a Morse function $f\colon M\to[0,1]$, it is well-known (cf. \cite{Mil}) that $M$ has the homotopy type of a CW-complex with one cell of dimension $k$ for each critical point of $f$ with index $k$. In particular, $M$ has the homotopy type of a collection of $n$-dimensional $k$-handles, one for each index-$k$ critical point of $f$, glued together. It is easy to see in this description that an index-$k$ critical point of $f$ is an index-$(n-k)$ critical point of $-f$. In this sense, a $k$-handle is thought of as an ``upside-down'' $(n-k)$-handle, and this can also be seen as $B^n = B^k\times B^{n-k} = B^{n-k}\times B^k$ from the definition of a $k$-handle.

For a Morse function $f\colon M\to[0,1]$, if the critical values of all index-$k$ critical points are less than the critical values of all index-$\ell$ critical points whenever $k<\ell$, then we say $f$ is \textit{self-indexing}. In the context of compact, orientable 3-manifolds, a self-indexing Morse function (with $\partial M\subset f^{-1}\{0,1\}$) defines a decomposition of $M$ into two compression bodies $V_0,V_1$, where $V_0$ contains the 0- and 1-handles and $V_1$ contains the 2- and 3-handles. Any regular level separating the 1-handles from the 2-handles is a closed surface $\Sigma\subset M$. The triple $(\Sigma;V_0,V_1)$ is called a \textit{Heegaard splitting} of $M$. We call $\Sigma$ a \textit{Heegaard surface} for $M$. We often abbreviate the triple $(\Sigma;V_0,V_1)$ as $\Sigma$ when the compression bodies are immaterial or clear from context. Because every compact, orientable 3-manifold $M$ admits a self-indexing Morse function \cite{Mil}, a Heegaard splitting $(\Sigma;V_0,V_1)$ always exists for $M$.

Let $(\Sigma;V_0,V_1)$ be a Heegaard splitting for the compact, orientable 3-manifold $M$. Let $\sigma_0$ and $\sigma_1$ be spines for $V_0$ and $V_1$, respectively. We see that $M\backslash(\sigma_0\cup\sigma_1)\cong\Sigma\times(0,1)$. We can use this product structure to define a function $h\colon M\to[0,1]$ where $h^{-1}(t)\cong\Sigma$ for all $0<t<1$ and $h^{-1}(i) = \sigma_i$ for $i = 1,2$. Such a function $h$ is called a \textit{sweep-out of $M$ induced by $\Sigma$} (or just a \textit{sweep-out} when $M$ and $\Sigma$ are clear from context).

\subsubsection{Surfaces in Handlebodies and Compression Bodies}

Let $\Delta$ be a collection of compressing disks for a compression body $W$. If compressing $\partial_-W$ in $W$ along all disks in $\Delta$ produces $\partial_+W$ and possibly a collection of $S^2$ components, then $\Delta$ is called a \textit{complete disk system for $W$}. If there are no $S^2$ components in the resulting compression, then $\Delta$ is called a \textit{minimal complete disk system} for $W$.

In a compression body $W$, let $\gamma$ be a curve in $\partial_+W$. Using the product structure of $W$, we stretch $\gamma$ across $W$ to $\partial_-W$ to create an annulus $\gamma\times[0,1]$. An annulus created in this fashion is called a \textit{spanning annulus} of $W$. If $\gamma$ is essential in $\partial_+W$, we say that $\gamma\times[0,1]$ is an \textit{essential spanning annulus} of $W$. See Figure \ref{fig:SpanningAnnuli}.

%%%%%%%%%%%%%%%%%%%%%%%%%%%
%% Spanning Annuli Image %%
%%%%%%%%%%%%%%%%%%%%%%%%%%%
\begin{figure}[h]
\centering
\includegraphics{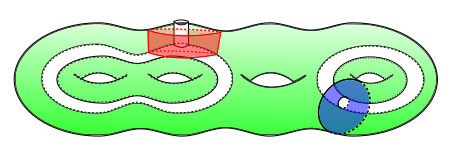}
\caption{An essential (blue) spanning annulus and a inessential, $\partial$-parallel spanning annulus (red) for a compression body. The boundary curves of the inessential $\partial$-parallel on the inner and outer boundary components of the compression body.}
\label{fig:SpanningAnnuli}
\end{figure}

\begin{theorem}[cf. \cite{Jaco}]
The only compact, connected, orientable, incompressible, $\partial$-incompressible, surfaces in a compression body that aren't parallel into the boundary are essential disks and essential spanning annuli. In a handlebody, such surfaces can only be essential disks.
\end{theorem}

%%%%%%%%%%%%%%%%%%%%%%%
%% The Curve Complex %%
%%%%%%%%%%%%%%%%%%%%%%%
\subsection{The Curve Complex}

\subsubsection{Constructing the Complex of Curves}

As outlined in the introduction, the following construction is due to Harvey \cite{Harv} as he studied the asymptotic geometry of a surface's Teichm\'uller space. To a compact, orientable (possibly disconnected) surface $F$, we assign a 1-dimensional simplicial complex $C(F)$ as follows:

\begin{enumerate}

	\item
	For each isotopy class of closed curves essential in $F$, add a vertex.
	
	\item
	Add an edge between distinct vertices if there are representatives of each isotopy class (one from each vertex) that are disjoint in $F$.

\end{enumerate}

This construction may be continued to construct a higher-dimensional simplicial complex by requiring that this complex be a flag complex. The result is called the \textit{curve complex} of $F$. We will only concern ourselves with the 1-skeleton $C(F)$ of this complex and still refer to it as the ``curve complex'' of $F$. See Figure \ref{fig:CurveComplex}.

%%%%%%%%%%%%%%%%%%%%%%%%%
%% Curve Complex Image %%
%%%%%%%%%%%%%%%%%%%%%%%%%
\begin{figure}[h]
\centering
\includegraphics[width=5in]{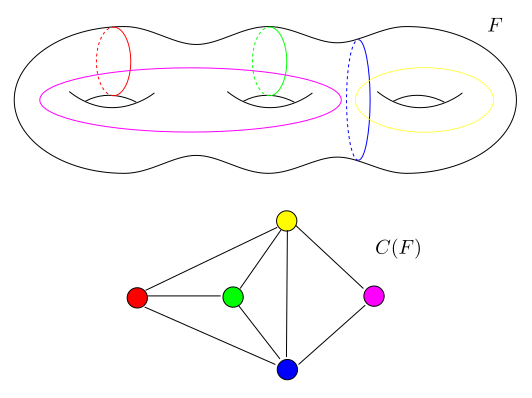}
\caption{A surface $F$ and the subgraph of its curve complex $C(F)$ induced by the vertices representing the colored curves.}
\label{fig:CurveComplex}
\end{figure}

For a pair of essential curves $\gamma_1,\gamma_2$ in the same connected component of $F$, we define the \textit{distance $d_\mathcal{C}(\gamma_1,\gamma_2)$} between $\gamma_1$ and $\gamma_2$ to be the graph distance in $C(F)$ between the vertices $[\gamma_1]$ and $[\gamma_2]$. If $\gamma_1$ and $\gamma_2$ lie in different connected components, then we take the convention $d_\mathcal{C}([\gamma_1],[\gamma_2]) = \infty$. We also use the convention that $d_\mathcal{C}([\gamma],\emptyset) = 0$ for any essential curve $\gamma\subset F$. We often abbreviate and abuse the notation $d_\mathcal{C}(\gamma_1,\gamma_2)$ to denote the distance between $\gamma_2$ and $\gamma_2$, omitting the brackets denoting isotopy classes of the curves.

\subsubsection{Types of Heegaard Splittings}

Let $\Sigma$ be a Heegaard splitting of a closed, orientable 3-manifold $M$. The surface $\Sigma$ is compressible in many ways to both sides in the handlebodies $H_1$ and $H_2$. Our goal is to measure how compressible $\Sigma$ is to both sides. We collect all compressing disks of $\Sigma$ in $H_1$ into the set $\Gamma_1$ and all compressing disks of $\Sigma$ in $H_2$ into the set $\Gamma_2$. In order to compare these sets, we appeal to the curve complex $C(\Sigma$) of the Heegaard surface; that is, where the two sets could potentially intersect.

Each disk in $\Gamma_1\cup\Gamma_2$ has a boundary curve $\gamma\subset\Sigma$ that is essential in $\Sigma$. The isotopy class $[\gamma]$ is represented by a vertex in $C(\Sigma)$. Denote by $V_i\subset C(\Sigma)$ the collection of vertices of $C(\Sigma)$ defined in this way by the set $\Gamma_i$ (for $i=1,2$).

The \textit{(Hempel) distance $d(\Sigma)$} of a Heegaard splitting $\Sigma$ \cite{Hemp} is the minimum graph distance $d_\mathcal{C}$ along $C(\Sigma)$ from an element of $[\gamma_1]\in V_1$ to an element of $[\gamma_2]\in V_2$; that is,
\[
d(\Sigma) = \min\{d_\mathcal{C}(\partial\alpha,\partial\beta) \ \mid \ \alpha\in\Gamma_1,\beta\in\Gamma_2\}.
\]

If $d(\Sigma) = 0$, there exist compressing disks of $\Sigma$ to both sides whose boundaries are isotopic in $\Sigma$. If we isotope these disks so that their boundary curves coincide, then the disks can be glued together to form a $S^2$ that intersects $\Sigma$ in a curve essential in $\Sigma$. In particular, the Heegaard splitting must be \textit{reducible} or is a genus-one Heegaard splitting of $S^1\times S^2$.

If $d(\Sigma)\leq1$, there exist compressing disks of $\Sigma$ to both sides whose boundaries can be isotoped to be disjoint. This implies the Heegaard splitting is \textit{weakly reducible} as defined in \cite{CG}; that is, there exist essential disks $D_i\subset V_i$ for $i=1,2$ such that $\partial D_1\cap\partial D_2 = \emptyset$. Observe that reducible Heegaard splittings are indeed weakly reducible splittings since we can isotope the boundary of one disk to be disjoint from the other. 

If $d(\Sigma)\geq2$, $\Sigma$ is \textit{strongly irreducible} \cite{CG}; that is, each essential disk $V_1$ intersects each essential disk in $V_2$.

Hartshorn \cite{Hart} and Scharlemann-Tomova \cite{ScTom} relate the existence of essential surfaces to strongly irreducible Heegaard surfaces, respectively, to the distance of a Heegaard splitting. We state these theorems below:

\begin{reptheorem}{thm:Hart}[\cite{Hart}]
Let $M$ be a Haken 3-manifold containing an orientable incompressible surface of genus $g$. Then any Heegaard splitting of $M$ has distance at most $2g$.
\end{reptheorem}

\begin{theorem}[\cite{ScTom}]
\label{thm:ScTom}
Suppose $P$ and $Q$ are Heegaard splittings of a closed 3-manifold $M$. Then either $d(P)\leq2g(Q)$ or $Q$ is isotopic to $P$ or to a stabilization of $P$.
\end{theorem}

%%%%%%%%%%%%%%%%%%%%%%%%%%%%%%%%%%%%%%
%% Circular Heegaard Decompositions %%
%%%%%%%%%%%%%%%%%%%%%%%%%%%%%%%%%%%%%%
\subsection{Circular Handle Decompositions}

In the following section, much of the notation and definitions follow from Manjarrez-Guti\'errez \cite{Fab}.

Let $K\subset S^3$ be a knot. Let $F\colon S^3\backslash K\to S^1$ be a Morse function and define $f\colon E(K)\to S^1$ to be the restriction $F|_{E(K)}$. Since a fundamental cobordism of $f$ can be isotoped to have no local maxima or minima \cite{Mil}, we assume that all critical points of $f$ have index 1 or 2.

We construct a handle decomposition of $E(K)$ from $f$ as follows: Choose $R$ a regular level of $f$ between an index-1 and an index-2 critical point. There are many such choices: We assume $R$ is chosen to have smallest genus among all choices. If $f$ has no critical points, then $K$ is fibered and there is only one choice for $R$. Otherwise, we see that the critical points of $f$ define collections $N = \{N_1,\ldots,N_k\}$ and $T = \{T_1,\ldots,T_k\}$ of 1- and 2-handles, respectively. We assume the handles in $N_1$ appear first after $R$ and, moreover, that the handles in $N_i$ appear before the handles in $T_i$ and that the handles in $T_i$ appear before the handles in $N_{i+1}$ (taking all indices modulo $k$ where necessary). Construct the compact manifold
\[
H = (R\times I)\cup(N_1\cup T_1)\cup\cdots\cup(N_k\cup T_k)
\]
by flowing $R$ along $E(K)$ via the gradient of $f$, attaching handles from $N$ and $T$ as prescribed by the critical points of $f$. Choose a regular level $S_i$ separating $N_i$ from $T_i$; that is,
\[
S_i \cong \overline{\partial[(R\times I)\cup(N_1\cup T_1)\cup\cdots\cup N_i]\backslash[\partial E(K)\cup(R\times\{0\})]}.
\]
Call $S_i$ a \textit{thick level} of $f$ and set $\mathcal{S} = \cup_{i=1}^kS_i$. Similarly, choose a regular level $F_i$ separating $T_i$ from $N_{i+1}$; that is,
\[
F_i = \overline{\partial[(R\times I)\cup(N_1\cup T_1)\cup\cdots\cup(N_i\cup T_i)]\backslash[\partial E(K)\cup(R\times\{0\})]}.
\]
Call $F_i$ a \textit{thin level} of $f$ and set $\mathcal{F} = \cup_{i=1}^kF_i$.

We also define
\[
W_i = (\text{collar of }F_i)\cap(N_i\cup T_i),
\]
which is a 3-manifold with boundary $\partial W_i = F_i\cup F_{i+1}\cup(W_i\cap\partial E(K))$. The thick level $S_i$ defines a compact (but not closed!) Heegaard surface for $W_i$, dividing it into compression bodies
\[
A_i = (\text{collar of }F_i)\cup N_i\quad\text{and}\quad B_i = (\text{collar of }S_i)\cup T_i.
\]
The boundary $\partial A_i$ can be seen as the union of three components; that is, $\partial A_i = S_i\cup F_i\cup \partial_vA_i$, where $\partial_vA_i = \partial A_i\cap\partial E(K)$. We call $\partial_vA_i$ the \textit{vertical boundary of $A_i$}. We can similarly define $\partial_vB_i = \partial B_i\cap \partial E(K)$ to be the vertical boundary of $B_i$. Note that the vertical boundary is an annulus.

Observe that $F_k$ is diffeomorphic to $R$. The function $f$ defines a diffeomorphism $\phi\colon R\to R$. When $K$ is fibered, $\phi$ is called the \textit{monodromy} of $K$. In this case, we see that
\[
E(K) = H/(R\times\{0\}) = \phi(R\times\{1\}).
\]
See Figure \ref{fig:CircularHandleDecomposition}. This construction is well-defined - a different choice of $R$ merely translates the labels of the sets $N$ and $T$ by the same amount.

\begin{define}
The collection $\mathcal{D} = \{(W_i;A_i,B_i)\}_{i=1}^k$ will be called the \textit{circular handle decomposition of $E(K)$ induced by $f$}.
\end{define}

%%%%%%%%%%%%%%%%%%%%%%%%%%%%%%%%%%%%%%%%%
%% Circular Handle Decomposition Image %%
%%%%%%%%%%%%%%%%%%%%%%%%%%%%%%%%%%%%%%%%%
\begin{figure}[h]
\centering
\includegraphics[width=5in]{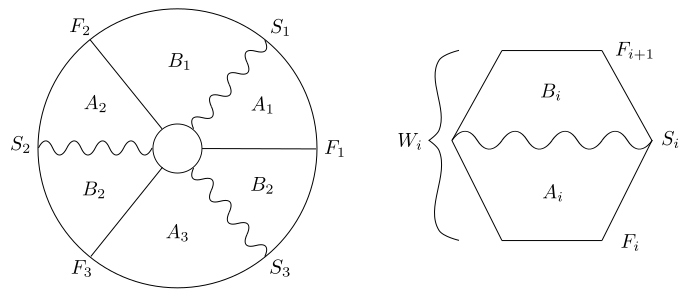}
\caption{A circular handle decomposition of a knot exterior and an arbitrary section $W_i$.}
\label{fig:CircularHandleDecomposition}
\end{figure}

\begin{define}
For a closed, connected surface $S\neq S^2$ define its \textit{complexity} $c(S) = 1-\chi(S)$. If $S$ has a nonempty boundary, define $c(S) = 1-\chi(\overline{S})$, where $\overline{S}$ denotes $S$ with its boundary components capped off with disks. We define $c(S^2) = 0$ and $c(D^2) = 0$. If $S$ is disconnected, define $c(S) = \sum c(S_i)$, where $S = \coprod_i S_i$ and each $S_i$ is connected.
\end{define}

\begin{define}
For a knot $K\subset S^3$ with circular handle decomposition $\mathcal{D}$ for $E(K)$, the \textit{circular width $cw(\mathcal{D})$ of the decomposition $\mathcal{D}$} is the multiset of integers $\{c(S_i)\}_{i=1}^k$, and $|cw(\mathcal{D})| = k$ is the number of thick levels in $\mathcal{D}$. The \textit{circular width $cw(E(K))$ of the knot exterior $E(K)$} is defined to be the minimal circular width among all circular handle decompositions of $E(K)$. The minimum is taken using the lexicographic ordering on multisets of integers.

The pair $(E(K),\mathcal{D})$ is in \textit{circular thin position} if $\mathcal{D}$ realizes the circular width of $E(K)$. When $K$ is a fibered knot, we define $cw(E(K)) = \emptyset$, so $|cw(E(K))| = 0$. If $|cw(E(K))| = 1$, we say that $K$ is \textit{almost-fibered}.
\end{define}

%%%%%%%%%%%%%%%%%%%%%%%%%%%%%%%%%%%%%%%%
%% Almost Fibered Decomposition Image %%
%%%%%%%%%%%%%%%%%%%%%%%%%%%%%%%%%%%%%%%%
\begin{figure}[h]
\centering
\includegraphics[height=2in]{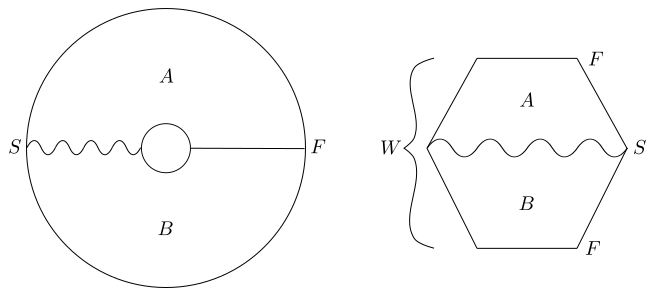}
\caption{An admissible circular handle decomopsition of a knot exterior and its only section $W$. The pair $(F,S)$ describes a circular Heegaard splitting of the knot exterior.}
\label{fig:AlmostFiberedDecomposition}
\end{figure}

Manjarrez-Guti\'errez \cite{Fab} examined knot exteriors using this setup in her doctoral dissertation. In addition to analyzing the behavior of circular width under some common knot operations, she showed:

\begin{theorem}[\cite{Fab}]
Let $K\subset S^3$ be a knot. At least one of the following holds:
\begin{enumerate}

	\item
	$K$ is fibered;
	
	\item
	$K$ is almost-fibered;
	
	\item
	$K$ contains a closed essential surface in its complement. Moreover, this closed essential surface is in the complement of an incompressible Seifert surface of $K$;
	
	\item
	$K$ has at least two non-isotopic, incompressible Seifert surfaces.

\end{enumerate}
\end{theorem}
The point we emphasize here is that circular thin position can be used to show the existence of multiple, non-isotopic, incompressible Seifert surfaces for a knot in $S^3$.

One other result from her work follows from Scharlemann-Thompson \cite{ST2} in their work on thin position for 3-manifolds:
\begin{theorem}[\cite{Fab}]
If $(E(K),\mathcal{D})$ is in circular thin position, then
\begin{enumerate}

\item
Each Heegaard splitting $S_i$ of $W_i$ is strongly irreducible.

\item
Each $F_i$ is incompressible in $E(K)$.

\item
Each $S_i$ is a weakly incompressible surface in $E(K)$.

\end{enumerate}
\end{theorem}

The converse of this theorem is not true in general. That is, a circular handle decomposition satisfying the three properties above need not be thin.

\begin{define}
A circular handle decomposition $\mathcal{D}$ is said to be \textit{locally thin} if the thin levels $F_i$ are incompressible and the thick levels $S_i$ are weakly incompressible.
\end{define}

A circular handle decomposition most resembles a Heegaard splitting when there is only one thick level and one thin level. We will focus on this case, which we call \textit{admissible}:

\begin{define}
Let $K\subset S^3$ be a knot with exterior $E(K)$, and let $\mathcal{D}$ be a circular handle decomposition of $E(K)$. We say that $\mathcal{D}$ is \textit{admissible} if $\mathcal{D} = \{(W;A,B)\}$; that is, $\mathcal{D}$ contains only one thick level $S$ and one thin level $F$. For an admissible decomposition, the pair $(F,S)$ will be called a \textit{circular Heegaard splitting of $E(K)$}.
\end{define}

\begin{remark}
If $K\subset S^3$ is a fibered knot, we can construct a circular Heegaard splitting $(F_0,F_1)$ of $E(K)$ using a parallel copy $F_1$ of the fiber surface $F_0$ and setting the thick level $S = F_1$. In this sense, fibered knots also admit circular Heegaard splittings even when the circle-valued Morse function has no critical points.
\end{remark}

%%%%%%%%%%%%%%%%%%%%%%%%%%%%%%%%%%%%%%%%%%%%%%%%%%%%%%%%%%%%%%%
%% Chapter 3 - Elementary Compressions and Circular Distance %%
%%%%%%%%%%%%%%%%%%%%%%%%%%%%%%%%%%%%%%%%%%%%%%%%%%%%%%%%%%%%%%%
\section{Elementary Compressions and Circular Distance}
\label{sec:CompsAndCircDist}

Let $M$ be a closed, orientable 3-manifold with a Heegaard splitting $(H_1,H_2;\Sigma)$. Suppose $M$ contains an essential, orientable surface $G$ of genus $g$. Hartshorn \cite{Hart} has shown that the distance of $\Sigma$ is bounded by twice the genus of $G$; that is, $d(\Sigma)\leq2g$. He uses what he calls an \textit{elementary compression} of $G\cap H_1$ into $H_2$ to define a sequence of isotopies of $G$ across $\Sigma$ that controls the distance of $\Sigma$. Since we will use these ideas extensively, we outline his argument below:

\begin{theorem}[\cite{Hart}]
\label{thm:Hart}
Let $M$ be a Haken 3-manifold containing an orientable incompressible surface of genus $g$. Then any Heegaard splitting of $M$ has distance at most $2g$.
\end{theorem}

\begin{proof}[Outline of proof:]
$ $\newline
\begin{enumerate}

	\item
	Isotope $G$ so that it intersects each $H_i$ ($i=1,2$) in incompressible, properly embedded components.

	\item
	Show that an elementary compression $G'$ of $G$ preserves the incompressibility of $G$ in $H_1$. Specifically, if each component of $G\cap H_1$, say, is incompressible in $H_1$, then so is each component of $G'\cap H_1$. We note that we may lose the incompressibility of some components of $G'\cap H_2$.
	
	\item
	Show that the distance in the curve complex $\mathcal{C}(\Sigma)$ between any curve in $G\cap\Sigma$ and any curve in $G'\cap\Sigma$ is at most one.
	
	\item
	Use a result of Kobayashi \cite{Kob} to isotope $G$ so that the components of $G\cap H_1$ contain exactly one essential disk.
	
	\item
	Use elementary compressions to produce a sequence of isotopies of $G$ such that the final embedding $G'$ is such that no component of $G'\cap H_1$ is a disk and that exactly one component of $G'\cap H_2$ is an essential disk. Performing these isotopies in such a minimal fashion bounds $d(\Sigma)$ by the number of elementary compressions in this sequence.
	
	\item
	Finally, show that the number of elementary compressions necessary for this kind of sequence of isotopies is no greater than $2g$.

\end{enumerate}

\end{proof}

We will adapt this method to circular Heegaard splittings of knot complements and the closed, essential surfaces they contain. We address the first three steps in this chapter while the last three steps are addressed in the next.

%%%%%%%%%%%%%%%%%%%%%%%%%%%%%%%%%%%%
%% Surfaces in Compression Bodies %%
%%%%%%%%%%%%%%%%%%%%%%%%%%%%%%%%%%%%
\subsection{Surfaces in Compression Bodies}

Let $K\subset S^3$ be a knot whose exterior $E(K)$ has a circular Heegaard splitting $(F,S)$. Suppose $G\subset E(K)$ is an essential, orientable surface of genus $g$. Consider its intersection $G\cap A$ with the compression body $A$. If there is a $\partial$-compressible component $G_0$ of $G\cap A$, then there exists a disk $D$ properly embedded in $A$ so that $D\cap A = \partial D = \alpha\cup\beta$, where $\alpha\subset G_0\cap A$ and $\beta\subset\partial A = (S\cup F)\cup(A\cap\partial E(K))$.

Suppose $\beta\cap\partial(E(K))=\emptyset$. Let $\eta(\alpha)\subset G_0$ be an open regular neighborhood of $\alpha$. Push $\overline{\eta(\alpha)}$ along $D$ and slightly into $B$. This effectively removes a (two-dimensional) 1-handle from $G_0$ and attaches a 1-handle in $B$ to $G_0$. The result is an isotopy of $G$. See Figure \ref{fig:ElementaryCompression}.

\begin{define}
The move just described is called a \textit{$\partial^*$-compression of $G$ from $A$ into $B$}. It is also known as an \textit{isotopy of type $A$} of $G$ across $S\cup F$.

If $G$ admits a $\partial^*$-compression from $A$ into $B$, we say $G$ is \textit{$\partial^*$-compressible from $A$ into $B$}.

If $G$ isn't $\partial^*$-compressible from $A$ into $B$, then we say $G$ is \textit{$\partial^*$-incompressible from $A$}.
\end{define}

\begin{remark}
Unless otherwise specified, a definition or claim about the compression body $A$ will also hold symmetrically for the compression body $B$.
\end{remark}

%%%%%%%%%%%%%%%%%%%%%%%%%%%%%%%%%%
%% Elementary Compression Image %%
%%%%%%%%%%%%%%%%%%%%%%%%%%%%%%%%%%
\begin{figure}[h]
\centering
\includegraphics[width=5in]{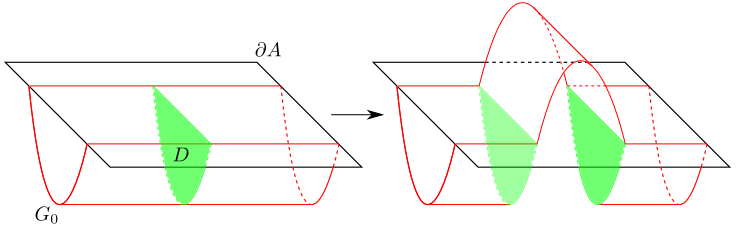}
\caption{A $\partial$-compression of $G$ from $A$ into $B$.}
\label{fig:ElementaryCompression}
\end{figure}

\begin{lemma}
\label{lem:EulerChar}
Suppose $G'$ is the result of a $\partial^*$-compression of $G$ from $A$ into $B$. Then $\chi(G'\cap A) = \chi(G\cap A) + 1$, and $\chi(G'\cap B) = \chi(G\cap B) - 1$.
\end{lemma}

\begin{proof}
Notice that, in the definition of $\partial^*$-compression, the removal of the neighborhood $\eta(\alpha)$ from $G_0$ removes an open disk and two edges from $G_0$. The net effect on $\chi(G_0)$ is an addition of 1. Similarly, for $G\cap B$, we have added that open disk and two edges to $G\cap B$. The net effect on $\chi(G\cap B)$ is the subtraction of 1.
\end{proof}

In light of the definition of distance of a Heegaard splitting in the previous chapter, we would like to omit any $\partial^*$-compressions that create inessential disk components. This happens only if $G\cap A$ has a $\partial$-parallel annulus component. However, since $\partial^*$-compressions avoid $\partial(E(K))$, this $\partial$-parallel annulus would be parallel into either $S$ or $F$ but not $\partial(E(K))$.

\begin{define}
Let $G\cap A$ contain an annular component $G_0$ that is parallel into either $S$ or $F$. This annulus will be called a \textit{$\partial^*$-parallel annulus}. A $\partial^*$-compression of $A_0$ will be called an \textit{annular compression from $A$ into $B$}.

Any $\partial^*$-compression of $G$ from $A$ into $B$ that is not an annular compression will be called an \textit{elementary compression}.
\end{define}

\begin{remark}
Unless otherwise stated, annular compressions of $G$ from $A$ into $B$ will always be followed by the isotopy carrying the resulting inessential disk across $S\cup F$. The net effect is to push the $\partial^*$-annulus out of $A$. In particular, $\chi(G\cap A) = \chi(G'\cap A)$ after an annular compression of $G$ from $A$ into $B$.
\end{remark}

The incompressible and $\partial$-incompressible surfaces in compression bodies have been classified \cite{BoOt}. These surfaces are $\partial$-parallel or they are essential disks or essential spanning annuli. We need a similar classification of incompressible and $\partial^*$-incompressible surfaces in the compression bodies defined by a circular Heegaard splitting of $E(K)$. This classification allows for one additional type of surface.

Recall that we can identify $A\cong(F\times I)\cup N$, where $N$ is a collection of 1-handles attached to $F\times\{1\}$. Choose an arc $\gamma$ properly embedded in $F\times\{1\}$ that is disjoint from the attaching disks for $N$. Thus, $\gamma\times I\subset A$ is a properly embedded disk $D$ in $A$ that is disjoint from $N$.

\begin{define}
The disk $D = \gamma\times I$ constructed above is called a \textit{product disk} of $A$. The arc $\gamma$ will be called a \textit{spanning arc} of $F$. We call $D$ \textit{inessential in A} if $\gamma$ is inessential in $F$. Otherwise, $D$ is said to be \textit{essential in $A$}. See Figure \ref{fig:ProductDisk}.

A product disk of $B$ is defined similarly by identifying $B\cong(S\times I)\cup T\cong(F\times I)\cup N'$, where $T$ is a collection of 2-handles and $N'$ is a collection of 1-handles dually equivalent to $T$.
\end{define}

%%%%%%%%%%%%%%%%%%%%%%%%
%% Product Disk Image %%
%%%%%%%%%%%%%%%%%%%%%%%%
\begin{figure}[h]
\centering
\includegraphics[width=5in]{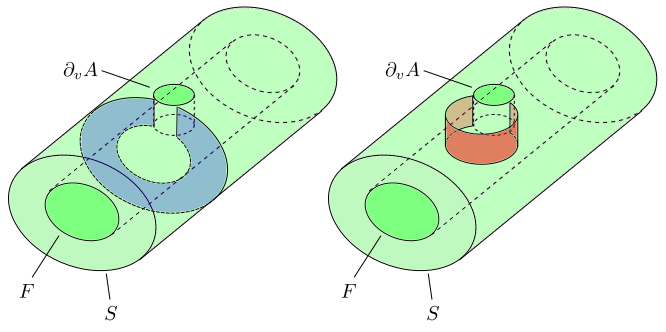}
\caption{A piece of $A$ showing an essential (blue) and inessential (red) product disk of $A$.}
\label{fig:ProductDisk}
\end{figure}

\begin{lemma}
\label{lem:DisksAndAnnuli}
Let $K\subset S^3$ be a knot whose exterior $E(K)$ has an admissible circular handle decomposition $\mathcal{D}$. Suppose $G$ is a compact, connected, non-$\partial$-parallel surface that is $\partial^*$-incompressible from $A$ and can be isotoped to be disjoint from $\partial_vA$. Then $G$ is either an essential disk or an essential spanning annulus.
\end{lemma}

\begin{proof}
The proof is similar to \cite{BoOt}. We include it here for completeness.

First isotope $G$ so that $G\cap\partial_vA=\emptyset$. Let $G$ be some component of $G\cap A$. For the sake of contradiction, suppose that $G_0$ is neither an essential disk nor an essential spanning annulus. Let $\Delta$ be a minimal, complete disk system for $A$ (this can be taken to be the co-cores of the 1-handles in the set $N$ of $\mathcal{D}$). Through a standard innermost-disk/outermost-arc argument, we may isotope $G_0$ so that $G_0\cap\Delta = \emptyset$. Cutting $A$ along $\Delta$ yields a manifold $V\cong F\times I$.

If $(\partial V)\cap G = \emptyset$, then $G$ is closed. This would imply that $G$ is compressible; hence, we conclude that $(\partial V)\cap G\neq\emptyset$. Without loss of generality, let's say that $F\cap G\neq\emptyset$.

Because $G_0$ is not a spanning annulus, we know that $\pi_1(G_0,S\cap G_0,\star)$ is non-trivial for some basepoint $\star\in S$. We also observe that $\pi_1(V,F)$ is trivial. These observations together show that there must exist an essential arc $\alpha$ properly embedded in $S$ cobounding a disk $D$ with some essential arc $\beta$ in $G_0$. In particular, $G_0$ is $\partial$-compressible by another innermost-disk/outermost-arc argument. Using this disk, either it is a $\partial$-compressing disk for $G_0$ or we can use it to find such a disk. Since we assumed that $G_0\cap\partial_vA = \emptyset$, this $\partial$-compression may be taken to be a $\partial^*$-compression.

From this contradiction, we conclude that $G_0$ can only be an essential disk or an essential spanning annulus.
\end{proof}

%\begin{corollary}
%\label{cor:ThickElemComps}
%Let $K\subset S^3$ be a knot whose exterior $E(K)$ has an admissible circular handle decomposition $(F,S)$. Suppose that $G$ is a compact, connected, orientable, incompressible surface properly embedded in $A$ such that $G\cap\partial_vA=\emptyset$, $G\cap F\neq\emptyset$, and $G\cap S\neq\emptyset$. If $G$ is $\partial^*$-compressible from $A$ into $B$ along a curve in $F$, then $G$ is $\partial^*$-compressible from $A$ into $B$ along a curve in $S$.
%\end{corollary}

%\begin{proof}
%The conditions that $G\cap\partial_vA = \emptyset$, $G\cap F\neq\emptyset$, and $G\cap S\neq\emptyset$ state that $G$ has at least one boundary component embedded in $\partial A\backslash\partial_vA$ and none in $\partial_vA$. Since $G$ is assumed to be $\partial^*$-compressible from $A$ into $B$ along a curve in $F$.
%\end{proof}

\begin{lemma}
\label{lem:ProdDisks}
Let $K\subset S^3$ be a knot whose exterior $E(K)$ has an admissible circular handle decomposition $\mathcal{D}$. Suppose $G$ is a compact, connected, non-$\partial$-parallel surface that is $\partial^*$-incompressible from $A$ and cannot be isotoped to be disjoint from $\partial_vA$. Then $G$ is an essential disk.
\end{lemma}

\begin{proof}
First isotope $G$ to intersect $\partial_vA$ in a minimal number of arcs. The intersection $\sigma = G\cap \partial_vA$ will be a collection of at least two parallel, essential spanning arcs of $\partial_vA$. Furthermore, we can isotope $\sigma$ to be vertical with respect to the identification $A\cong (F\times I)\cup N$. Specifically, each arc in $\sigma$ map be considered to be identified with $\{x\}\times I$ for some $x\in\partial F$. Unless otherwise specified, we now take all isotopies of $G$ to fix $\sigma$.

Now, as in the previous lemma, we may find a minimal, complete disk system $\Delta$ for $A$. We then use a standard innermost-disk/outermost-arc argument to isotope $G$ disjoint from $\Delta$ so that $G$ lies in a manifold $V\cong F\times I$.

Choose a collection of arcs $\Gamma$ properly embedded in $F\times\{1\}$ that are disjoint from the attaching disks of the 1-handles in $N$. Furthermore, choose $\Gamma$ so that $G\backslash\eta(\Gamma)$ is a disk. Observe that each arc $\gamma_i\in\Gamma$ is a spanning arc for some essential product disk $D_i$ of $A$. These resulting product disks $\mathcal{E}=\{D_i\}_{i=1}^N$ constitute a minimal complete (product) disk system for the handlebody $V$.

For sake of contradiction, we assume $G$ is not a product disk. Then $G$ is either (i) $\chi(G)\leq0$ or (ii) a disk that is not a product disk. Suppose $\chi(G)\leq0$ and that the genus of $G$ is at least one. Then we can choose an essential arc in $G$ with both endpoints in $S$. An argument similar to Lemma \ref{lem:DisksAndAnnuli} shows that there is a $\partial^*$-compression disk for $G$ from $A$. This contradiction implies that the genus of $G$ is zero.

In this case, $\partial G$ is comprised of at least two components; otherwise, we are in case (ii) above. Then $G$ contains an essential arc whose endpoints lie in different components of $\partial G$ and in the same surface, either $F$ or $S$. Once again, an argument similar to Lemma \ref{lem:DisksAndAnnuli} yields a $\partial^*$-compression disk for $G$ from $A$. Therefore, $G$ must have only one boundary component; that is, $G$ is a disk so that $\chi(G) = 1$. We conclude that $G$ can only be a disk.
\end{proof}

\begin{corollary}
\label{cor:SurfsInCompBods}
Let $K\subset S^3$ be a knot whose exterior $E(K)$ has an admissible circular handle decomposition $(F,S)$. The only compact, connected, orientable, properly embedded surfaces in the compression body $A$ that are not $\partial$-parallel, are incompressible, and are $\partial^*$-incompressible from $A$ are essential disks and essential spanning annuli.
\end{corollary}

%%%%%%%%%%%%%%%%%%%%%%%%%%%%%%%%%%%%%%%%%
%% Behavior of Elementary Compressions %%
%%%%%%%%%%%%%%%%%%%%%%%%%%%%%%%%%%%%%%%%%
\subsection{Behavior of Elementary Compressions}

Recall the setting at the beginning of Hartshorn's argument. There is an incompressible surface $G$ of genus $g$ inside a 3-manifold $M$, and $M$ has a Heegaard splitting $(H_1,H_2,;\Sigma)$. Hartshorn established that $\partial$-compressions of $G\cap H_1$ from $H_1$ don't alter the incompressibility of the components of $G\cap H_1$. We show the same of $G\cap A$ and $G\cap B$ if the circular Heegaard splitting is circular locally thin.

Our first objective is to show that an essential surface $G\subset E(K)$ can be isotoped to intersect $S\cup F$ in curves that are essential in both $S$ and $F$. This is not immediately obvious due to the fact that $S$ is weakly incompressible.

\begin{lemma}
\label{lem:MinInt}
Let $K\subset S^3$ be a knot whose exterior $E(K)$ has a locally thin circular handle decomposition $\mathcal{D}$. Suppose that $E(K)$ contains a closed, orientable incompresible surface $G$ that intersects $S\cup F$ in a minimal number of curves. Then

\begin{enumerate}[(i)]

	\item
	$G\cap A$ and $G\cap B$ are incompressible in $A$ and $B$, respectively.
	
	\item
	$G\cap(S\cup F)$ is a collection of simple closed curves that are essential in $S\cup F$.
	
	\item
	there are no $\partial$-parallel annulus components of either $G\cap A$ or $G\cap B$. In particular, there are no $\partial^*$-parallel annuli in $G\cap A$ or $G\cap B$.

\end{enumerate}

\end{lemma}

\begin{proof}
The proof is nearly identical to \cite{Hart}. We include the argument here for completeness.

\begin{enumerate}[(i)]

	\item
	Without loss of generality, assume that $G\cap A$ has a component compressible in $A$. Then there is a disk $D$ properly embedded in $A$ such that $\partial D$ doesn't bound a disk in $G\cap A$. Because $G$ is incompressible in $E(K)$, there is a second disk $D'$ in $G$ such that $\partial D' = \partial D$. Observe here that $D'\cap(S\cup F)\neq\emptyset$ since $\partial D$ doesn't bound a disk in $G\cap A$. Because $E(K)$ is irreducible, the disks $D$ and $D'$ cobound a ball in $E(K)$. We can then isotope $G$ so that $D'$ may be pushed through this ball and entirely out of $A$, into $B$, and off of $S\cup F$. This reduces the number of components in $G\cap(S\cup F)$, contradicting the assumption of minimality.
	
	\item
	Without loss of generality, suppose that $c$ is a curve in $G\cap\partial A$ that is inessential in $\partial A$. Then there is a disk $D$ in $\partial A$ with $\partial D = c$. Because $G$ is incompressible in $E(K)$, there is another disk $D'$ in $G$ such that $\partial D' = c$ as well.
	
	Consider the surface $G' = (G\backslash D')\cup D$. Since $E(K)$ is irreducible, we see that $G$ is actually isotopic to $G'$. Hence, we may push $D$ from $G'$ slightly off of $S\cup F$ to make $G'$ intersect $S\cup F$ in fewer components.
	
	\item
	Without loss of generality, assume that $A$ has a $\partial$-parallel annulus component $G_0$. Because $G$ is closed, $G\cap\partial E(K) = \emptyset$ so that $G\cap\partial_vA = \emptyset$. Hence, any $\partial$-parallel component is actually $\partial^*$-parallel. Therefore, any such component can be pushed out of $A$ and into $B$ so as to reduce the number of components in $G\cap(S\cup F)$.

\end{enumerate}

\end{proof}

Now that $G$ intersects $S\cup F$ in such a desirable fashion, we would like to show that elementary compressions preserve this structure. Moreover, we need to show that they exist in the first place.

\begin{lemma}
\label{lem:ElemComps}
Let $K\subset S^3$ be a knot whose exterior $E(K)$ has an admissible, locally thin circular handle decomposition $(F,S)$. Suppose that $G\subset E(K)$ is a closed, connected, orientable, incompressible surface. If $G$ intersects $S\cup F$ in a minimal number of curves and $\chi(G\cap A)<0$, then there is an elementary compression of $G\cap A$ from $A$ into $B$.
\end{lemma}

\begin{proof}
Since $G$ is incompressible in $E(K)$, it follows that $G\cap\partial A\neq\emptyset$. By Lemma \ref{lem:MinInt}, $G\cap A$ is incompressible in $A$. If all components of $G\cap A$ are disks and annuli, then $\chi(G\cap A)\geq0$. Our assumption that $\chi(G\cap A)<0$ implies that there is some component of $G\cap A$ that is neither a disk nor an annulus. In particular, $G\cap A$ has a $\partial$-compressible component. If this $\partial$-compression isn't a $\partial^*$-compression, we can assume that this component is $\partial^*$-incompressible (otherwise, we are done). By Lemma \ref{lem:DisksAndAnnuli}, this component must be either a disk or an annulus. Neither such surface is $\partial$-compressible since no component of $G\cap A$ is a $\partial^*$-parallel annulus. We conclude then that there must be a $\partial^*$-compression of some component of $G\cap A$. Because this $\partial^*$-compression doesn't take place on a $\partial^*$-parallel annulus, it is an elementary compression of $G$ from $A$ into $B$. 
\end{proof}

\begin{lemma}
\label{lem:CompactElemComps}
Let $K\subset S^3$ be a knot whose exterior $E(K)$ has an admissible circular handle decomposition $(F,S)$. Suppose that $G$ is a compact, orientable, incompressible surface properly embedded in $A$ with no $\partial^*$-parallel components such that $G\cap\partial_vA$ has at most one component and that $G\cap\partial A\neq\emptyset$. If $\chi(G\cap A)<0$, then there is an elementary compression of $G\cap A$ from $A$ into $B$.
\end{lemma}

\begin{proof}
If $G\cap A$ has no such elementary compression, then by \ref{cor:SurfsInCompBods} each of its components is either an essential disk or essential spanning annulus. But then we would have $\chi(G\cap A)\geq0$, and this contradicts our hypothesis.
\end{proof}

\begin{lemma}
\label{lem:NiceComps}
Let $K\subset S^3$ be a knot whose exterior $E(K)$ has an admissible circular handle decomposition $(F,S)$. Suppose that $G$ is a compact, orientable surface properly embedded in $E(K)$ so that $G\cap A$ is incompressible in $A$ and that each component of $G\cap(S\cup F)$ is essential in $\partial A$. Denote by $G'$ the embedding of $G$ after an elementary compression from $A$ into $B$. Then:

\begin{enumerate}[(i)]

	\item
	$G'$ is also incompressible in $A$.
	
	\item
	The components of $G'\cap A$ are also essential in $\partial A$.

\end{enumerate}

\end{lemma}

\begin{proof}
This proof is nearly identical to that found in \cite{Hart}. We include it here for completeness.

\begin{enumerate}[(i)]

	\item
	Suppose $G$ is compressed in $A$ via the disk $D\subset A$ with boundary $\partial D = \alpha\cup\beta$, where $\alpha\subset G$ and $\beta\subset\partial A\backslash\partial_vA$. We can choose a neighborhood $D\times I$ of this disk in such a way that $(\partial D)\times I = (\alpha\cup\beta)\times I$ keeps $\alpha\times I\subset G$ and $\beta\times I\subset\partial A\backslash\partial_vA$. The elementary compression of $G$ along $D$ replaces $\alpha\times I$ with two disks $D\times(\partial I)$. Name these disks $D_0 = D\times\{0\}$ and $D_1 = D\times\{1\}$. We can then consider $D_0$ and $D_1$ as submanifolds of $G'$.
	
	Assume that $G'$ is compressible in $A$, and let $c\subset G'$ be the boundary curve of some compressing disk $D'\subset A$. We can isotope $c$ to be disjoint from $D_0$ and $D_1$ so that an innermost-disk argument moves $D'$ disjoint from $D_0$ and $D_1$ as well. Reversing the elementary compression from above, we can view $D'$ now as a compressing disk from $G$, thereby contradicting the incompressibility of $G\cap A$.
	
	\item
	Let $\Sigma = \partial A = S\cup F$. Suppose instead that there is a component $c'\subset G'\cap\Sigma$ that bounds a disk $D\subset\Sigma$. Then $c'$ must come from an elementary compression of $G$ from $A$ into $B$; otherwise, $c'$ is an inessential curve in $G\cap\Sigma$ and we contradict our assumption that $G\cap(\Sigma)$ was a collection of curves essential in $\Sigma$. Thus, we take $\beta\subset\Sigma$ to be the defining arc of this $\partial^*$-compression. The proof now proceeds by cases dependent on the number of components of $G\cap\Sigma$ that $\beta$ joins.
	
	\textit{One component} - If $\beta$ joins only one component of $G\cap\Sigma$, the compression breaks $c$ into two components. One of these components is the curve $c'$ from above, and we name the other curve $c''$. Either the disk $D$ contains $c''$ or it doesn't.
	
	We first assume that $c''\subset D$. Then $c''$ itself must bound a disk in $D$. Hence, $c'$ is isotopic to the original curve $c$ so that $c$ is inessential in $\Sigma$. This contradiction implies that $c''$ lies outside of the disk $D$. If this is the case, we can isotope $D$ so that its boundary can be decomposed into two arcs in $\Sigma$: the arc $\beta$ from the elementary compression and the arc $\gamma = \overline{c'\backslash\beta}$. If $D'$ is the disk realizing the compression of $G\cap \Sigma$ with $\partial D' = \alpha\cup\beta$ ($\alpha\in G$), then the union $D'' = D\cup D'$ is a disk such that $\partial D'' = \alpha\cup\gamma\subset G$. If we push the interior of $D''$ slightly into $A$, then $D''$ constitutes a compressing disk for $G$ in $A$. This contradicts the incompressibility of $G$ in $A$. See Figure \ref{fig:NiceCompressions1Component}.
	
%%%%%%%%%%%%%%%%%%%%%%%%%%%%%%%%%%%%%%%%%
%% Nice Compressions 1 Component Image %%
%%%%%%%%%%%%%%%%%%%%%%%%%%%%%%%%%%%%%%%%%
\begin{figure}[h]
\centering
\includegraphics[width=5in]{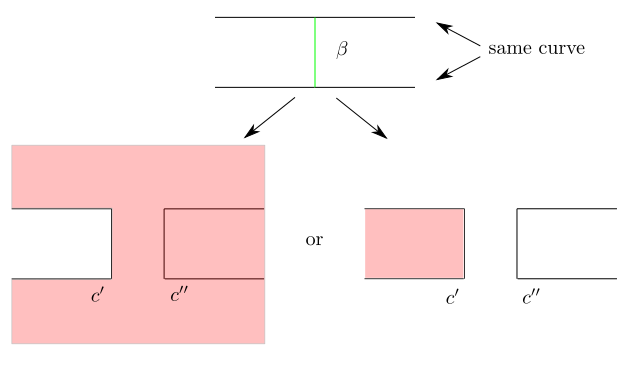}
\caption{The local result of compressing $G$ along the curve $\beta$ connecting a single component of $\partial G$. The disk $D$ mentioned in the proof of Lemma \ref{lem:NiceComps} is shaded red.}
\label{fig:NiceCompressions1Component}
\end{figure}
	
	\textit{Two components} - If $\beta$ instead joins two components $c_0, c_1$ of $G\cap\Sigma$, then observe that $c'$ is the only curve of $G'\cap\Sigma$ affected by the compression. There are two ways that $c'$ can bound $D$: either $D$ contains $\beta$ or it doesn't.
	
	If $D$ contains $\beta$, then reversing the compression pinches $D$ into two disks with boundaries $c_0$ and $c_1$, respectively. Hence, both $c_0$ and $c_1$ are contained inside a disk and themselves bound disks. This contradicts our assumption that $c_0$ and $c_1$ were both essential in $\Sigma$.
	
	If $\beta$ lies outside $D$, then reversing the compression is the same as gluing a (two-dimensional) 1-handle onto $D$, thereby creating two boundary components (namely, $c_0$ and $c_1$). That is, $D$ is converted into an annulus. From our arguments in (i) above, this annulus must be inessential in $A$ so that it is $\partial$-parallel. Hence, the $\partial^*$-compression must have been used instead for an annular compression rather than our assumed elementary compression. See Figure \ref{fig:NiceCompressions2Components}.
	
%%%%%%%%%%%%%%%%%%%%%%%%%%%%%%%%%%%%%%%%%%
%% Nice Compressions 2 Components Image %%
%%%%%%%%%%%%%%%%%%%%%%%%%%%%%%%%%%%%%%%%%%
\begin{figure}[h]
\centering
\includegraphics[width=5in]{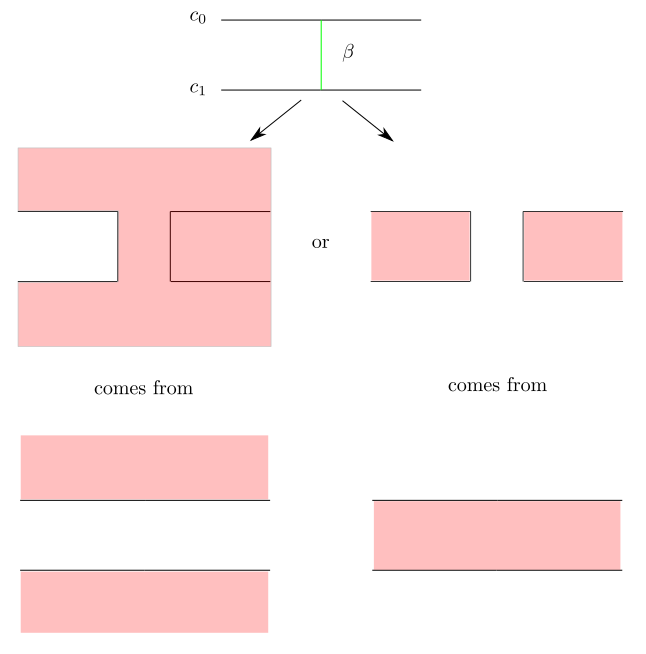}
\caption{The local result of compressing $G$ along the curve $\beta$ connecting two components $c_0,c_1\subset\partial G$. The disk $D$ mentioned in the proof of Lemma \ref{lem:NiceComps} is shaded red.}
\label{fig:NiceCompressions2Components}
\end{figure}

\end{enumerate}

\end{proof}

%%%%%%%%%%%%%%%%%%%%%%%
%% Circular Distance %%
%%%%%%%%%%%%%%%%%%%%%%%
\subsection{Circular Distance}

The necessity of the requirement that the intersections $G\cap(S\cup F)$ be essential in $S\cup F$ stems from our utilization of the curve complex of $S\cup F$. In the sequence of isotopies that we eventually create, we need to be able to compare intersection curves from one term of the sequence to the intersection curves of the next. The following lemma makes this idea more precise.

\begin{lemma}
\label{lem:NiceDistances}
Let $K\subset S^3$ be a knot whose exterior $E(K)$ has an admissible circular handle decomposition $(F,S)$. Suppose that $G$ is a compact, orientable surface properly embedded in $E(K)$. Denote by $G'$ the embedding of $G$ after a $\partial^*$-compression from $A$ into $B$, and let $c\subset G\cap(S\cup F)$ and $c'\subset G'\cap(S\cup F)$ be curves essential in $S\cup F$. Then $d_\mathcal{C}(c,c')\leq 1$.
\end{lemma}

\begin{proof}
The argument again follows very closely to \cite{Hart}. We include it here for completeness.

Call $\Sigma = S\cup F$. Let $D$ be the disk that realizes the $\partial^*$-compression with $\partial D = \alpha\cup\beta$, where $\alpha\subset G\cap A$ and $\beta\subset\Sigma$.

If the $\partial^*$-compression is annular, then $c$ is entirely removed from $S\cup F$ as a result. Hence, any $c'$ chosen from $G'\cap\Sigma$ must be disjoint from $c$ so that $d_\mathcal{C}(c,c') = 1$. We now assume that the $\partial^*$-compression is, in fact, an elementary compression of $G$ from $A$ into $B$.

If $c$ is not affected by the compression, then $c$ can be made disjoint from $c'$ so that $d_\mathcal{C}(c,c')\leq1$. Therefore, we assume that $c$ is indeed affected by the compression. We proceed by cases dependent on whether the arc $\beta$ joins one component or two different components of $G\cap\Sigma$. 

\textit{Two components} - If $\beta$ joins two components $c_0,c_1$ of $G\cap\Sigma$, then we can choose closed collar neighborhoods $\eta_0$ and $\eta_1$, respectively, so that $\beta\cap\text{int}(\eta_0\cup\eta_1) = \emptyset$. Let $\hat c_0$ and $\hat c_1$ be the boundary components of $\eta_0$ and $\eta_1$, respectively, where $\beta\cap(\hat c_0\cup\hat c_1) = \emptyset$. Then $\hat c_0$ is isotopic to $c_0$ and $\hat c_1$ is isotopic to $c_1$, and the curve $\delta$ resulting from the compression is disjoint from both $\hat c_0$ and $\hat c_1$. Hence, because $c$ is either $c_0$ or $c_1$, any component $c'$ of $G'\cap\Sigma$ can be chosen so that it is disjoint from $c$. Therefore, $d_\mathcal{C}(c,c')\leq1$. See Figure \ref{fig:NiceDistances2Components}.

%%%%%%%%%%%%%%%%%%%%%%%%%%%%%%%%%%%%%%%
%% Nice Distances 2 Components Image %%
%%%%%%%%%%%%%%%%%%%%%%%%%%%%%%%%%%%%%%%
\begin{figure}[h]
\centering
\includegraphics[width=5in]{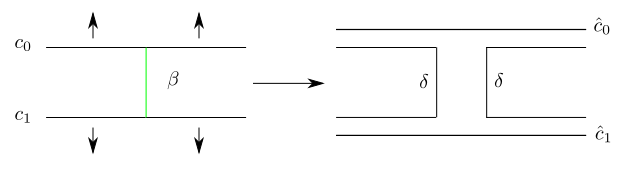}
\caption{The local result of compressing $G$ along the curve $\beta$ connecting two components $c_0,c_1\subset\partial G$.}
\label{fig:NiceDistances2Components}
\end{figure}

\textit{One component} - If $\beta$ joins the same component $c$ of $G\cap\Sigma$, then we consider a normal push-off $\mathcal{G}$ of $G$ so that $G\cap\mathcal{G} = \emptyset$. In particular, $\mathcal{G}$ can be chosen so that it is disjoint from the $\partial^*$-compression disk $D$. Denote by $c_\epsilon\subset\mathcal{G}\cap\Sigma$ the image of $c$ under the push-off.

We then perform the elementary compression of $S$ from $A$ into $B$. The curve $c$ is thereby pinched into the pair of curves $c'_0,c'_1\subset G'\cap\Sigma$, both of which are now disjoint from $c_\epsilon$. Hence, $c$ is isotopic to a curve that is disjoint from any component $c'$ of $G'\cap\Sigma$. Thus, $d_\mathcal{C}(c,c') = 1$. See Figure \ref{fig:NiceDistances1Component}.

%%%%%%%%%%%%%%%%%%%%%%%%%%%%%%%%%%%%%%
%% Nice Distances 1 Component Image %%
%%%%%%%%%%%%%%%%%%%%%%%%%%%%%%%%%%%%%%
\begin{figure}[h]
\centering
\includegraphics[width=5in]{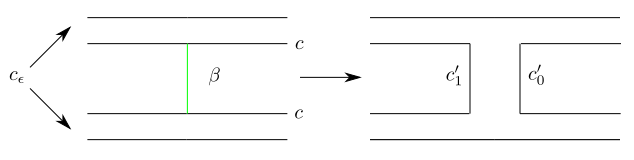}
\caption{The local result of compressing $G$ along the curve $\beta$ connecting two components $c_0,c_1\subset\partial G$.}
\label{fig:NiceDistances1Component}
\end{figure}

\end{proof}

\begin{define}
\label{def:CircDist}
Let $K\subset S^3$ be a knot whose exterior $E(K)$ has an admissible circular handle decomposition $(F,S)$. Denote by $\Gamma_A$ the set of all essential disks and essential spanning annuli of $A$. Similarly, define the set $\Gamma_B$ for $B$.

For the circular Heegaard splitting $(F,S)$, we define its \textit{circular distance} to be
\[
cd(F,S) = \min\{d_\mathcal{C}(\partial_S\alpha,\partial_S\beta) + d_\mathcal{C}(\partial_F\alpha,\partial_F\beta) \ \mid \ \alpha\in\Gamma_A, \beta\in\Gamma_B\}.
\]
and the \textit{thick distance of $S$} to be
\[
td(S) = \min\{d_\mathcal{C}(\partial_S\alpha,\partial_S\beta) \ \mid \ \alpha\in\Gamma_A,\beta\in\Gamma_B\}.
\]
\end{define}

In the case of a fibered knot $K\subset S^3$, we immediately find $td(S) = 0$ for any circular Heegaard splitting $(F,S)$ of $E(K)$. This is realized by a vertical, essential annulus $\mathcal{A}$ in $E(K)\backslash\eta(F_0)$. This annulus is cut by $S$ into a pair of essential spanning annuli $\alpha\subset A$ and $\beta\subset B$ with $\partial_S\alpha = \partial_S\beta$. Hence, the only non-zero contribution to the circular distance $cd(F,S)$ comes from its thin level $F$.
	
From a theorem of Johnson's \cite{Jesse}, we know that $d_\mathcal{C}(\partial_F\alpha,\partial_F\beta)\leq4$. Hence, this annulus $\mathcal{A}$ also gives an upper bound to any circular Heegaard splitting $(F,S)$ of $E(K)$ of a fibered knot:
	
\begin{lemma}
\label{lem:FiberedDistance}
Let $K\subset S^3$ be a fibered knot and $(F,S)$ be a circular Heegaard splitting for its exterior $E(K)$. Then $cd(F,S)\leq4$.
\end{lemma}

\begin{remark}
Because we have this bound, we will now only consider those knots in $S^3$ that are not fibered. In particular, \textit{the results that follow do not necessarily hold for fibered knots}.
\end{remark}

When we remove the essential spanning annuli and product disks from $\Gamma_A$ and $\Gamma_B$ and regard $S$ as a Heegaard splitting of $E(K)\backslash\eta(F)$, the thick distance $td(S)$ is the usual Hempel distance $d(S)$. In some sense, we can view the thick distance as a generalization of Hempel distance. As such, the usual notions of reducibility and weak reducibility of Heegaard splittings are extended to include essential spanning annuli and essential product disks. Specifically, for a circular Heegaard splitting $(F,S)$, we say that $(F,S)$ is:
	
\begin{itemize}

	\item
	\textit{reducible} if $cd(F,S) = 0$,
		
	\item
	\textit{weakly reducible} if $cd(F,S) = 1$, and
		
	\item
	\textit{strongly irreducible} if $cd(F,S)\geq2$.
	
\end{itemize}
	
For non-fibered knots, we also note that we have the inequalities
	\[
	0\leq td(S)\leq cd(F,S)\leq d(S)\leq td(S) + 2.
	\]
The last inequality follows since, for any essential spanning annulus in a compression body, there is an essential disk that is disjoint from it.

%%%%%%%%%%%%%%%%%%%%%%%%%%%%%%%%
%% Bounding Circular Distance %%
%%%%%%%%%%%%%%%%%%%%%%%%%%%%%%%%
\section{Bounding Circular Distance}

This chapter begins with the proof of the Main Theorem \ref{thm:Main}. We then go on to adapt this theorem for the case of incompressible Seifert surfaces. We end with the proof of Theorem \ref{thm:SeifertMain}, which gives a partial affirmation of a conjecture of Manjarrez-Guti\'errez (Remark 3.6 of \cite{Fab}) suspecting that minimal-genus Seifert surfaces of a knot $K\subset S^3$ always arise as part of a thin circular handle decomposition (necessarily as thin levels) of the exterior $E(K)$.

%%%%%%%%%%%%%%%%%%%%%%%%%%%%%%%
%% Proof of the Main Theorem %%
%%%%%%%%%%%%%%%%%%%%%%%%%%%%%%%
\subsection{Proof of the Main Theorem}

In this section, we complete our proof of the analog to Hartshorn's theorem for an incompressible surface $G\subset E(K)$ and a circular Heegaard splitting $(F,S)$ of $E(K)$. We recall that we must isotope $G$ so that $G\cap A$ contains exactly one essential disk or essential spanning annulus $\alpha$. Our goal is to then isotope $G$ across $S$ so that $G\cap B$ contains exactly one essential disk or essential spanning annulus $\beta$. Then $\alpha$ and $\beta$ realize an upper bound for the circular distance $cd(F,S)$.

To compute this bound explicitly, our isotopy of $G$ across $S$ must control the distances of the curves of intersection $G\cap(S\cup F)$. If $G'$ is the result of an elementary compression of $G\cap A$ from $A$ into $B$, then Lemma \ref{lem:NiceDistances} says the distance between a curve $c'\in G'\cap(S\cup F)$ and a curve $c\in G\cap(S\cup F)$ is at most one. Hence, the upper bound we compute is exactly the number of elementary compressions of $G\cap A$ from $A$ into $B$ across $S$ in order to produce the essential disk or essential spanning annulus $\beta$.

Our first goal, then, is to produce the essential disks or essential spanning annuli $\alpha$ and $\beta$.

\begin{lemma}
\label{lem:Yes}
Let $K\subset S^3$ be a knot whose exterior $E(K)$ has an admissible circular handle decomposition $(F,S)$ such that $F$ is incompressible and $td(S)\geq2$. Suppose that $G$ is a closed, connected, orientable, essential surface in $E(K)$ such that each component of $G\cap(S\cup F)$ is essential in either $S$ or $F$ and that each component of $G\cap A$ is incompressible in $A$. If $G\cap B$ contains an essential disk or an essential spanning annulus, then there is a sequence of isotopies
\[
G\simeq G_0\simeq G_1\simeq\cdots\simeq G_k\simeq G_{k+1}\simeq\cdots\simeq G_n
\]
of $G$ such that
\begin{itemize}

	\item
	Each component of $G_i\cap A$ is incompressible in $A_i$ for each $0\leq i\leq n$;
	
	\item
	Each component of $G_i\cap(S\cup F)$ is essential in either $S$ or $F$;
	
	\item
	For any choice of components $c_i\in G_i\cap(S\cup F)$ that belong to the same component of $S\cup F$, $d_\mathcal{C}(c_i,c_{i+1})\leq 1$ for $0\leq i\leq n-1$.
	
	\item
	For $0\leq i\leq k$, at least one component of $G_i\cap B$ is an essential disk or essential spanning annulus.
	
	\item
	For $k+1\leq i\leq n-1$, no component of either $G_i\cap A$ or $G_i\cap B$ is an essential disk or essential spanning annulus;
	
	\item
	The final isotopy from $G_{n-1}$ to $G_n$ ensures that $G_n\cap A$ contains exactly one essential disk or essential spanning annulus component;
	
	\item
	We must have $k\leq n-2$.

\end{itemize}

\end{lemma}

\begin{proof}
First, remove from $G\cap A$ any $\partial^*$-parallel annuli via annular compressions to form $G_0$. If none exist, we take $G_0 = G$. Because $td(S)\geq2$, no component of $G\cap A$ is an essential disk or essential spanning annulus so that $\chi(G_0\cap A) < 0$. Some component of $G_0$ must meet $S$ in a curve essential in $S$ since, otherwise, $G$ may be passed through $S$ so as to lie entirely in $A$ or $B$. Hence, there is an elementary compression of $G_0\cap A$ from $A$ into $B$ across $S$. Performing this elementary compression creates $\hat{G}_1$, an isotopy of $G_0$ that differs only by an elementary compression. Now remove from $\hat{G}_1\cap A$ any $\partial^*$-parallel annuli via annular compressions. As before, if none exist, we taken $G_1 = \hat{G}_1$. Continue in this fashion to create the remaining $G_i$ for $2\leq i\leq k$.

We choose $k$ to be the greatest integer such that $G_k\cap B$ contains an essential disk or essential spanning annulus. Such an integer exists by Lemma \ref{lem:EulerChar}. Starting with $G_k$, continue the procedure above to create $G_i$ for $k+1\leq i\leq n$, choosing $n$ to be the smallest integer such that $G_n\cap A$ contains an essential disk or essential spanning annulus component. We note here that $k\leq n-2$ because $td(S)\geq2$. The first three bullet points are shown using Lemmas \ref{lem:MinInt} and \ref{lem:NiceDistances}.

If $G_n\cap A$ contains any essential disk components, then there is only one such component; otherwise, $\chi(G_n\cap A)\geq\chi(G_{n-1}\cap A)+2$ which contradicts Lemma \ref{lem:EulerChar}.

It is possible, however, that $G_n\cap A$ contains two essential spanning annuli. If this is indeed the case, then these annuli must have come from a pair-of-pants component of $G_{n-1}\cap A$. Observe that there is an alternative elementary compression of this component across $F$. See Figure \ref{fig:SaddleChange}. Performing this elementary compression produces a single essential spanning annulus in $G_n\cap A$. This shows the fourth, fifth, and sixth bullet points.

%%%%%%%%%%%%%%%%%%%%%%%%%
%% Saddle Change Image %%
%%%%%%%%%%%%%%%%%%%%%%%%%
\begin{figure}[h]
\centering
\includegraphics[width=5in]{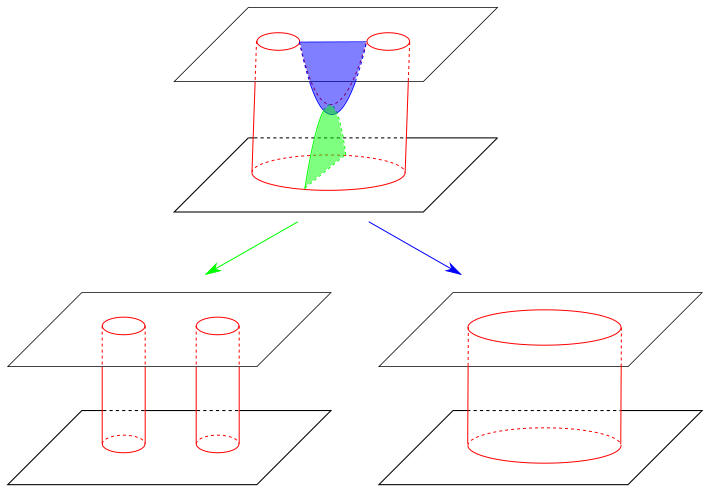}
\caption{If there is an elementary compression through $F$ coming from a pair of pants, then there is an alternative elementary compression through $S$ that we make instead.}
\label{fig:SaddleChange}
\end{figure}

To show the final bullet point, we notice that $G_i\cap A$ and $G_i\cap B$ cannot both contain an essential disk and an essential annulus since $td(S)\geq2$ and $F$ is incompressible. If $k = n-1$, there must be components $c_k\subset G_k\cap S$ and $c_n = c_{k+1}\subset G_n\cap S$ such that $c_k$ bounds an essential disk or an essential spanning annulus in $G_n\cap A$. The third bullet point shows that $d_\mathcal{C}(c_k,c_n)\leq1$. Hence, we would have $td(S)\leq1$ and contradict our assumption that $td(S)\geq2$. Then $k\leq n-2$ and we have shown the final bullet point.
\end{proof}

Again, we emphasize the symmetry between $A$ and $B$ in this lemma. That is, we could have started with an essential disk or essential spanning annulus in $G\cap A$ instead. We could then perform our sequence of isotopies via elementary compressions of $G$ from $B$ into $A$. We also take a moment here to observe that either $\chi(G_n\cap A)\leq1$ or $\chi(G_n\cap B)\leq1$.

If no such essential disk or essential annulus exists in either $G\cap A$ or $G\cap B$, then we can still produce one in the same fashion as above.

\begin{lemma}
\label{lem:No}
Let $K\subset S^3$ be a knot whose exterior $E(K)$ has an admissible circular handle decomposition $(F,S)$ such that $F$ is incompressible and $td(S)\geq2$. Suppose that $G$ is a closed, connected, orientable, essential surface in $E(K)$ such that each component of $G\cap(S\cup F)$ is essential in either $S$ or $F$ and that each component of $G\cap A$ is incompressible in $A$. If neither $G\cap A$ nor $G\cap B$ contain an essential disk or an essential spanning annulus, then there is a sequence of isotopies
\[
G\simeq G_{-m}\simeq G_{1-m}\simeq\cdots\simeq G_0 \simeq G_1\simeq\cdots\simeq G_n
\]
of $G$ such that
\begin{itemize}

	\item
	Each component of $G_i\cap A$ is incompressible in $A_i$ for $-m\leq i\leq0$ and $G_i\cap B$ is incompressible in $B_i$ for $0\leq i\leq n$;
	
	\item
	Each component of $G_i\cap(S\cup F)$ is essential in either $S$ or $F$;
	
	\item
	For any choice of components $c_i\in G_i\cap(S\cup F)$ that belong to the same component of $S\cup F$, $d_\mathcal{C}(c_i,c_{i+1})\leq 1$ for $-m\leq i\leq n-1$.
	
	\item
	There is exactly one component of $G_{-m}\cap A$ and $G_n\cap B$ that is either an essential disk or an essential spanning annulus;
	
	\item
	For $1-m\leq i\leq n-1$, no component of either $G_i\cap A$ or $G_i\cap B$ is an essential disk or essential spanning annulus.

\end{itemize}

\end{lemma}

\begin{proof}
Since both $G\cap A$ and $G\cap B$ contain no essential disks or essential spanning annuli, it follows that both $\chi(G\cap A) < 0$ and $\chi(G\cap B) < 0$. Hence, $G\cap A$ has an elementary compression from $A$ into $B$ across $S$. We define the sequence of isotopies as before to get the $G_i$ for $-m\leq i\leq-1$. Then $G_{-m}\cap A$ contains exactly one essential disk or essential spanning annulus. Now starting at $G_0$, but reversing the roles of $A$ and $B$ in the previous lemma, gives us the $G_i$ surfaces for $1\leq i\leq n$. Then $G_n\cap B$ contains exactly one essential disk or essential spanning annulus. All the noted properties of the sequence are now satisfied via the proof of Lemma \ref{lem:Yes}.
\end{proof}

\begin{theorem}[Main Theorem]
\label{thm:Main}
Let $K\subset S^3$ be a knot whose exterior $E(K)$ has a circular Heegaard splitting $(F,S)$ such that $F$ is incompressible. If its exterior $E(K)$ contains a closed, orientable, essential surface $G$ of genus $g$, then $cd(F,S) \leq 2g$.
\end{theorem}

\begin{proof}
We may isotope $G$ so that it intersects $S\cup F$ in a minimal number of curves. We divide the argument into two cases: (i) $td(S)\geq2$ and (ii) $td(S)\leq1$

We first assume that $td(S)\geq2$. By Lemma \ref{lem:MinInt}, this embedding of $G$ satisfies the conditions of Lemmas \ref{lem:Yes} and \ref{lem:No}. From these lemmas, we conclude that there must exist a sequence of isotopies $G_0\simeq\cdots G_n$ such that $G_0\cap A$ and $G_n\cap B$ both contain exactly one essential disk or essential spanning annulus each. Call these components $P_A\subset G_0\cap A$ and $P_B\subset G_n\cap B$, respectively. We also choose a sequence of curves $c_i\subset G_i\cap (S\cup F)$ such that $c_0\in\partial P_A$ and $c_n\in\partial P_B$

We bound $n$ by first observing that $\chi(G) = \chi(G_0\cap A) + \chi(G_0\cap B)$ because $G\simeq G$ and $G\cap (S\cup F)$ is a collection of circles with Euler characteristic zero. Since $G_0\cap A$ contains exactly one essential disk or essential spanning annulus, we use Lemma \ref{lem:EulerChar} to show inductively that

\begin{eqnarray*}
\chi(G)	&	=			&	\chi(G_0\cap A) + \chi(G_0\cap B)	\\
2 - 2g	&	\leq	&	1 + \chi(G_0\cap B)	\\
1 - 2g	&	\leq	&	\chi(G_0\cap B)	\\
1 - 2g	&	\leq	&	\chi(G_n\cap B) - n	\\
1 - 2g	&	\leq	&	1-n	\\
n				&	\leq	&	2g.
\end{eqnarray*}

Since $n$ was chosen to be the smallest integer such that $G_n\cap B$ contains an essential disk or essential spanning annulus, we then need $n$ elementary compressions of $G$ to move from having an essential disk or essential spanning annulus in $A$ to having one in $B$. Hence, we can bound the circular distance of the decomposition as
\[
cd(F,S)	\leq	d_\mathcal{C}(c_0,c_n)	\leq	n	\leq	2g.
\]

We now assume that $td(S)\leq1$. The remarks following Definition \ref{def:CircDist} show that $cd(F,S)\leq3$. If $cd(F,S)\leq2$, then $cd(F,S) \leq 2g$ trivially. Hence, the only other case we need to consider is when $cd(F,S)=3$ with $g=1$. We show now that this is impossible.

Observe that $td(S) = 1$ and $d(S) = 3$ when $cd(F,S)=3$ and $S$ is strongly irreducible. Suppose that $G\cap B$ contains an essential disk or essential spanning annulus. Then we are able to recover the sequence of isotopies constructed in Theorem \ref{lem:Yes} until we reach the final paragraph of the proof. The last bullet point follows from the fact that $cd(F,S) = 3$; that is, $G_i\cap A$ and $G_i\cap B$ cannot both contain essential disks and essential spanning annuli for any $0\leq i\leq n$. Hence, $n>k$.

If $k = n-1$, there would exist components $c_k\subset G_k\cap S$ and $c_n\subset G_n\cap S$ such that $c_k$ bounds an essential disk or essential spanning annulus $P_B\subset G_n\cap B$ and that $c_n$ bounds an essential disk or essential spanning annulus $P_A\subset G_k\cap A$. Moreover, because $td(S)=1$, we see that $d_\mathcal{C}(c_k,c_n)=1$.

The strong irreducibility of $S$ shows $P_A$ and $P_B$ cannot both be disks. A similar contradiction in $cd(F,S) = 3$ is found, without loss of generality, if $P_A$ is a disk and $P_B$ is an annulus. We then see that $P_A$ and $P_B$ are both annuli. However, $d_\mathcal{C}(\partial_FP_A,\partial_FP_B)\leq1$ so that $cd(F,S)\leq2$ again. We conclude that $k\leq n-2$.

If neither $G\cap A$ nor $G\cap B$ contain an essential disk or essential spanning annulus, then we are able to recover the sequence constructed in Theorem \ref{lem:No} in its entirety. The Euler characteristic argument above then shows that $cd(F,S)\leq2$, thereby contradicting the requirement that $cd(F,S)=3$.
\end{proof}

\begin{corollary}
\label{cor:Main}
Let $K\subset S^3$ be a knot whose exterior $E(K)$ has a circular Heegaard splitting $(F,S)$ such that $F$ is incompressible. If its exterior $E(K)$ contains a closed, orientable, essential surface $G$ of genus $g$ that can be isotoped to be disjoint from $F$, then $d(S) \leq 2g$.
\end{corollary}

\begin{proof}
The proof of Theorem \ref{thm:Main} holds. Moreover, we can realize the bound on $cd(F,S)$ without using essential spanning annuli.
\end{proof}

%%%%%%%%%%%%%%%%%%%%%%%%%%%%%%%%%%%%%%%%%%%%%%%%%%%%%%%%%%%%%%
%% Circular Distance Bound via an Alternate Seifert Surface %%
%%%%%%%%%%%%%%%%%%%%%%%%%%%%%%%%%%%%%%%%%%%%%%%%%%%%%%%%%%%%%%
\subsection{Circular Distance Bound via an Alternate Seifert Surface}

It was conjectured in \cite{Fab} that it may be possible that Seifert surfaces of minimal genus appear as the thin levels of thin circular handle decompositions. We provide a partial affirmation of this conjecture; however, because Seifert surfaces are not closed, more justification is needed in order to prove this fact. We first state a uniqueness theorem as a corollary to the following theorem:

\begin{theorem}
\label{thm:SeifertMain}
Let $K\subset S^3$ be a knot whose exterior $E(K)$ has a circular Heegaard splitting $(F,S)$ such that $F$ is incompressible. If $F'$ is an incompressible Seifert surface for $K$ with genus $g$ that is not isotopic to $F$, then $d(S)\leq2g+1$.
\end{theorem}

\begin{corollary}
\label{cor:UniqueSeifert}
Let $K\subset S^3$ be a knot whose exterior $E(K)$ has a circular Heegaard splitting $(F,S)$ such that $F$ is a Seifert surface of $K$ that realizes the Seifert genus $g(K)$. If $d(S)>2g(K)+1$, then $F$ is the unique Seifert surface of minimal genus for $K$ up to isotopy.
\end{corollary}

We adapt our work from the previous section to the case where the closed surface $G\subset E(K)$ is instead taken to be properly embedded with non-empty boundary. Section \ref{sec:CompsAndCircDist} classifies the compact, connected, incompressible, $\partial^*$-incompressible surfaces that may appear in the compression bodies $A$ and $B$, and Lemma \ref{lem:CompactElemComps} indicates when elementary compressions exist.

The next step is to mimic Theorems \ref{lem:Yes} and \ref{lem:No}. We do this by replacing the closed surface $G$ with an incompressible Seifert surface $F'$ for the knot $K$. We first notice that Lemma \ref{lem:MinInt} is still applicable in this setting; that is, we can still isotope $F'$ to (i) intersect $S$ in curves that are essential in $S$ and (ii) intersect $A$ and $B$ so that $F'\cap A$ and $F'\cap B$ are incompressible in $A$ and $B$, respectively. In addition, Scharlemann and Thompson \cite{ST1} show that $F'$ may be assumed to be disjoint from the thin level $F$ (as $F$, too, is an incompressible Seifert surface for $K$).

Observe that if $F'$ is disjoint from $F$, then $F'\cap A$ and $F'\cap B$ contain no essential spanning annuli. Hence, we first consider the case where one of either $F'\cap A$ or $F'\cap B$ contains an essential disk component, and then we consider the case where neither $F'\cap A$ nor $F'\cap B$ contain any essential disks.

\begin{lemma}
\label{lem:SeifertYes}
Let $K\subset S^3$ be a non-fibered knot whose exterior $E(K)$ has a circular Heegaard splitting $(F,S)$ such that $F$ is incompressible. Suppose further that $F'$ is an incompressible Seifert surface for $K$ disjoint from and non-isotopic to $F$ such that $F'\cap S$ is a collection of simple closed curves that are essential in $S$.

If $F'\cap A$ contains an essential disk and each component of $F'\cap A$ is incompressible in $A$, then there exists a sequence of isotopies
\[
F'\simeq F'_0\simeq F'_1\simeq\cdots\simeq F'_m\simeq\cdots\simeq F'_n
\]
such that

\begin{itemize}

	\item
	Each component of $F'_i\cap S$ is essential in $S$;

	\item
	Each component of $F'_i\cap A$ is incompressible in $A$;
	
	\item
	For any choice of components $c_i\in F'_i\cap S$ and $c_{i+1}\in F'_{i+1}\cap S$, we have $d_\mathcal{C}(c_i,c_{i+1})\leq1$ for $0\leq i\leq n-1$;
	
	\item
	$F'_n\cap B$ contains exactly one essential disk component;
	
	\item
	$F'_i\cap A$ contains an essential disk component for each $0\leq i\leq m$, and neither $F'_i\cap A$ nor $F'_i\cap B$ contain any essential disk components for $m+1\leq i\leq n-1$.
	
\end{itemize}

\end{lemma}

\begin{proof}
If $F'\cap A$ contains any $\partial^*$-parallel annuli, use annular compressions to remove them from $B$ and define the resulting surface $F'_0$; if not, then define $F'_0 = F'$. Because $td(S)\geq2$ and $F'_0\cap F = \emptyset$, we see that $F'_0\cap B$ contains no essential disks or essential spanning annuli. This imples that $F'_0\cap B$ is $\partial^*$-compressible. Perform an elementary compression to form $\hat{F}'_1$. Now proceed inductively (as in the proof of Theorem \ref{lem:Yes}) and let $m$ be the smallest integer such that either (i) $F'_m\cap B$ contains an essential disk component or (ii) $\chi(F'_m\cap B) = 0$ and $F'_m$ contains no essential disk components.

If $F'_m\cap B$ contains an essential disk component, recall that $\chi(F'\cap B) = \chi(F'_{m-1}\cap B) + 1$ by Lemma \ref{lem:EulerChar}. If $F'_m\cap B$ were to contain more that one essential disk component, then it would be the case that $\chi(F'_m\cap B\geq\chi(F'_{m-1}\cap B) + 2$. This contradicition dictates that $m=n$ in this case.

If $\chi(F'_m\cap B) = 0$ and $F'_m\cap B$ contains no essential disk components, then $F'_m\cap B$ must be a collection of annuli. Note that none of these annuli are $\partial^*$-parallel by construction. If all of these annuli intersect $\partial_vB$, they would also intersect $S$ since $F'_m\cap F = \emptyset$. We may then isotope $F'_m$ to lie entirely in $A$. We then see that $F'_m$ is a properly embedded surface in $B$ that is connected, incompressible, and disjoint from both $S$ and $F$. Additionally, $\partial F'_m\subset\partial_vB$ has exactly one component. Therefore, $F'_m$ must be isotopic to $F$, which contradicts the assumption of the lemma.

We can now choose an annulus in $F'_m\cap B$ that is disjoint from $\partial_vB$. Being disjoint from $F$, this annulus must be $\partial^*$-compressible. We perform the elementary compression to create $\hat{F}'_{m+1}$ and then form $F'_{m+1}$ by removing any $\partial^*$-parallel annuli via annular compressions as before. Set $n = m+1$. We then find $\chi(F'_n\cap B) = 1$ so that $F'_n\cap B$ contains exactly one essential disk component.

It cannot be the case that both $F'\cap A$ nor $F'\cap B$ contain essential disk components as $S$ is strongly irreducible. We conclude $n>m$. If $m = n-1$, then there exist curves $c_m\subset F'_m\cap S$ and $c_{m+1} = c_n\subset F'_n\cap S$ so that $c_m$ bounds an essential disk in $F'_m\cap A$ and $c_n$ bounds an essential disk in $F'_n\cap B$. Then $d(S)\leq1$ so that this contradiction allows us to conclude that $m\leq n-2$.

The remaining points of the lemma follow as in the proof of Lemma \ref{lem:Yes}.
\end{proof}

\begin{lemma}
\label{lem:SeifertNo}
Let $K\subset S^3$ be a knot whose exterior $E(K)$ has a circular Heegaard splitting $(F,S)$ such that $F$ is incompressible. Suppose further that $F'$ is an incompressible Seifert surface for $K$ disjoint from and non-isotopic to $F$ such that $F'\cap S$ is a collection of simple closed curves that are essential in $S$.

If neither $F'\cap A$ nor $F'\cap B$ contains any essential disk components, and if each component of $F'\cap A$ and $F'\cap B$ is essential in $A$ and $B$, respectively, then there exists a sequence of isotopies
\[
F'_{-m}\simeq F'_{-m+1}\simeq\cdots\simeq F'_0\simeq F'\simeq F')1\simeq\cdots\simeq F'_n
\]
such that

\begin{itemize}

	\item
	Each component of $F'_i\cap S$ is essential in $S$;
	
	\item
	Each component of $F'_i\cap A$ and $F'_i\cap B$ is essential in $A$ and $B$, respectively;
	
	\item
	For any choice of components $c_i\in F'_i\cap S$ and $c_{i+1}\in F'_{i+1}\cap S$, we have $d_\mathcal{C}(c_i,c_{i+1})leq1$ for $-m\leq i\leq n-1$;
	
	\item
	Both $F'_{-m}\cap A$ and $F'_n\cap B$ contain exactly one essential disk component, and no $F'_i\cap A$ or $F'_i\cap B$ contains any essential disk components for $-m+1\leq i\leq n-1$.

\end{itemize}

\end{lemma}

\begin{proof}
This is nearly identical to the proof of Lemma \ref{lem:No}.
\end{proof}

We now prove Theorem \ref{thm:SeifertMain}:

\begin{proof}[Proof of Theorem \ref{thm:SeifertMain}]
This proof is similar to the Main Theorem \ref{thm:Main}. If $d(S)\leq1$, then the theorem follows; hence, we assume $d(S)\geq2$.

First, isotope $F'$ so that it intersects $S$ in a minimal number of components. Regardless of whether or not $F'\cap A$ and $F'\cap B$ contain essential disk components, Lemmas \ref{lem:SeifertYes} and \ref{lem:SeifertNo} prove there exists a sequence $F'_0,\ldots,F'_n$ of Seifert surfaces isotopic to $F'$ such that $F'_0\cap B$ and $F'_n\cap A$ each contain exactly one essential disk component. Moreover, $F'_i\cap S$ is a collection of simple closed curves that are essential in $S$ for all $0\leq i\leq n$, and $\chi(F'_{i+1}\cap A) = \chi(F'_i\cap A) + 1$.

In any case, let $c_B$ be the boundary of the essential disk in $F'_0\cap B$ and $c_A$ be the boundary of the essential disk in $F'_n\cap A$. For $0\leq i\leq n-1$, we have $d_\mathcal{C}(c_i,c_{i+1})\leq1$. By setting $c_0 = c_B$ and $c_n = c_A$, repeated application of the triangle inequality shows that $d_\mathcal{C}(c_A,c_B)\leq n$ (cf. the proof of the Main Theorem). Moreover, we find

\begin{eqnarray*}
\chi(F')	&	=			&	\chi(F'_0\cap A) + \chi(F'_0\cap B)	\\
1 - 2g		&	\leq	&	1 + \chi(F'_0\cap B)	\\
-2g				&	\leq	&	\chi(F'_0\cap B)	\\
-2g				&	\leq	&	\chi(F'_n\cap B) - n	\\
-2g				&	\leq	&	1-n	\\
n					&	\leq	&	2g+1.
\end{eqnarray*}

so that
\[
d(S)\leq d_\mathcal{C}(c_A,c_B)\leq n\leq 2g + 1.
\]
\end{proof}

%%%%%%%%%%%%%%%%%%%%%%%%%%%%%%%%%%%%%%%%%%%%%%%%%%%%%%%%%%%%%%%%%%%%%%%%%%%%%%%%%%%%%%%%%%%%%%%%%%%%%%%%%%%%%%%%%%%%%%%%%%%%%%%%%%%%%%%%%%%%%%%%%%%%%%%%%%%%%%%%%%%%%%%%%%%%%%%%%%%%%%%%%%%%%%%%%%%%%%%%%%%%%%%%%%%%%%%%%%%%%%%%%%%%%%%%%%%%%%%%%%%%%%%%%%%%%%%%%%%%%%%%%%%%%%%%%%%%%%%%%%%%%%%%%%%%%%%%%%%%%%%%%%%%%%%%%%%%%%%%%%%%%%%%%%%%%%%%%%%%%%%%%%%%%%%%%%%%%%%%%%%%%%%%%%%%%%%%%%%%%%%%%%%%%%%%%%%%%%%%%%%%%%%%%%%%%%%%%%%%%%%%%%%%%%%%%%%%%%%%%%%%%%%%%%%%%%%%%%%%%%%%%%%%%%%%%%%%%%%%%%%%%%%%%%%%%%%%%%%%%%%%%%%%%%%%%%%%%%%%%%%%%%%%%%%%%%%%%%%%%%%%%%%%%%%%%%%%%%%%%%%%%%%%%%%%%%%%%%%%%%%%%

%%%%%%%%%%%%%%%%%%%%%%%%%%%%%%%%%%%%%%%%%%%%%%%%%%%%%%%%%%%%%%%%%%%%%%%%%%%%%%%%%%%%
%%%%%%%%%%%%%%%%%%%%%%%%%%%%%%%%% END PAPER %%%%%%%%%%%%%%%%%%%%%%%%%%%%%%%%%%%%%%%%
%%%%%%%%%%%%%%%%%%%%%%%%%%%%%%%%%%%%%%%%%%%%%%%%%%%%%%%%%%%%%%%%%%%%%%%%%%%%%%%%%%%%

%%%%%%%%%%%%%%%%%%
%% Bibliography %%
%%%%%%%%%%%%%%%%%%

\end{document}